\documentclass[11pt]{article}
\usepackage{amsmath,amsfonts,amssymb,amsthm}
\usepackage{epsfig,color,graphics}
\usepackage{graphicx}
\usepackage{paralist}
\usepackage{tikz}

\numberwithin{equation}{section}

\newtheorem{theorem}{Theorem}[section]

\newtheorem{proposition}[theorem]{Proposition}
\newtheorem{lemma}[theorem]{Lemma}
\newtheorem{corollary}[theorem]{Corollary}

\newcommand{\V}{\operatorname{Var}}
\newcommand{\C}{\operatorname{Cov}}

\def\R{\mathbb{R}}
\def\N{\mathbb{N}}
\def\Z{\mathbb{Z}}

\def\P{\mathbb{P}}
\def\1{\mathbf{1}}

\def\cX{{\mathcal X}}

\def\I{\1}

\DeclareMathOperator{\card}{card}

\newcommand{\BP}{\mathbb{P}}
\newcommand{\BE}{\mathbb{E}}
\newcommand{\BQ}{\mathbb{Q}}
\newcommand{\BM}{\mathbb{M}}

\newcommand{\bN}{{\mathbf N}}

\newcommand{\bX}{{\mathbf X}}
\newcommand{\bY}{{\mathbf Y}}

\newcommand{\bM}{{\mathbf M}}
\newcommand{\bG}{{\mathbf G}}

\newcommand{\llra}{\longleftrightarrow}
\newcommand{\lra}{\leftrightarrow}

\begin{document}

\title{The random connection model and functions of edge-marked Poisson processes: second order properties and normal approximation}

\date{\today}
\renewcommand{\thefootnote}{\fnsymbol{footnote}}

\author{G\"unter Last\footnotemark[1]\,, Franz Nestmann\footnotemark[1]\,
and Matthias Schulte\footnotemark[2]}

\footnotetext[1]{guenter.last@kit.edu, franz.nestmann2@kit.edu,
Karlsruhe Institute of Technology}
\footnotetext[2]{matthias.schulte@stat.unibe.ch, University of Bern}

\maketitle

\begin{abstract}
\noindent
The random connection model is a random graph
whose vertices are given by the points of
a Poisson process and whose edges are obtained by randomly connecting
pairs of Poisson points in a position dependent but independent way.
We study first and second order properties of the numbers of components isomorphic
to given finite connected graphs. For increasing observation windows in an Euclidean setting we
prove qualitative multivariate and quantitative univariate
central limit theorems for these component counts as well as a qualitative central limit theorem
for the total number of finite components.
To this end we first derive general results
for functions of edge marked Poisson processes, which we believe to be of independent interest.
\end{abstract}

\maketitle

\noindent\textbf{Key words:} Random connection model, component count, covariance structure, central limit theorem, Poisson process, edge marking, Gilbert graph,
random geometric graph

\noindent\textbf{MSC 2010 subject classifications:} 60D05, 60F05, 05C80, 60G55

\section{Introduction}\label{secintro}

For many decades random graphs have attracted much interest because of both their mathematical beauty and the importance of complex networks, i.e., large graphs with a highly non-trivial structure, in many other sciences; see \cite{Bollobas,vdHofstad,JansonLuczakRucinski} and the references therein. Two prominent models for random graphs are \emph{Erd\H os-R\'enyi graphs} and \emph{random geometric graphs}.  In the Erd\H os-R\'enyi model, which goes back to \cite{ErdoesRenyi59,Gilbert59} and has been studied extensively ever since (see e.g.\ \cite{Bollobas,vdHofstad,JansonLuczakRucinski}), the vertex set is $[n]:=\{1,\hdots,n\}$ for some $n\in\N$, and pairs of distinct vertices are independently connected by an edge with probability $p\in[0,1]$. While the Erd\H os-R\'enyi graph is a  purely combinatorial object, the vertices of a random geometric graph are points in $\R^d$ and are given by a random point sample or, more precisely, a point process $\zeta$ on $\R^d$. Two distinct vertices $x,y\in\zeta$ are connected by an edge whenever their Euclidean distance $|x-y|$ does not exceed a given threshold $r>0$. The random geometric graph was introduced in \cite{Gilbert61} and is also called Gilbert graph. For a comprehensive investigation of random geometric graphs we refer to the monograph \cite{Penrose03}.

The {\em random connection model} (RCM) studied in the present paper can be seen as a combination of the Erd\H os-R\'enyi graph and the random geometric graph and is also known as \emph{soft random geometric graph}. As in case of random geometric graphs we start with an underlying point process $\zeta$ on $\R^d$. Between two distinct vertices $x,y\in\zeta$ we draw an edge with the probability $\varphi(x,y)$ depending on the positions of $x$ and $y$ in $\R^d$. For different pairs of vertices these decisions are made independently. For the choice $\varphi(x,y):=\mathbf{1}\{|x-y|\leq r\}$, $x,y\in\R^d$, with $r>0$ the vertices are connected in a deterministic way and the resulting graph is a random geometric graph. In case that $\varphi(x,y):=p\in(0,1)$, $x,y\in\R^d$, the decisions which vertices are connected are not affected by the locations of the points of $\zeta$ at all. If, additionally, $\zeta$ consists of a fixed finite number of points, the resulting graph is isomorphic to an Erd\H os-R\'enyi graph. Because of these extreme choices for $\varphi$, where the edges are completely determined by the geometry of $\zeta$ or are completely independent of the geometry of $\zeta$, one can think of the RCM as an interpolation between the Erd\H os-R\'enyi graph and the random geometric graph.
The results of this paper will include the random geometric graph as a special case but not the Erd\H os-R\'enyi graph.

In this paper the underlying point process $\zeta$ will always be a Poisson process. The independence properties of Poisson processes are important for our analysis of the RCM and, in contrast to a binomial point process consisting of a fixed number of i.i.d.\ points, this choice allows to construct a stationary RCM in $\R^d$. As far as we know, the RCM with an underlying (stationary) Poisson process was first considered in \cite{Penrose91}, where the percolation behaviour and first-order properties of components were studied. For further percolation results involving more general underlying stationary point processes we refer to \cite{BurtonMeester93,MeesterRoy96}.
The connectivity of RCMs on finite Poisson processes, the closely related number of isolated vertices and more general degree counts were studied in \cite{DettGeorg16,GilesGeorgiouDettmann,Iyer2018, MaoAnderson, Penrose16, Penrose17} (see also the references therein), while the diameter was investigated in \cite{DevroyeFraimanDiameter}.
Some of this work was motivated by applications in wireless communication networks; see e.g.\ \cite{FranMees08}.
The RCM can also be seen as a continuous version of discrete long-range percolation models where edges between the points of $\mathbb{Z}^d$ are drawn according to a connection function
$\varphi$; see \cite{DeijfenHofstadHooghiemstra,DeprezWuethrich,GrimmettKeaneMastrand} and the references therein.

The general RCM (defined and studied in Sections 2--4)
is based on a Poisson process $\eta$ on a Borel space $(\bX,\mathcal{X})$
with diffuse $\sigma$-finite intensity measure $\lambda$. We interpret
$\eta$ as a random discrete subset of $\bX$.
Fix a measurable and symmetric {\em connection function} $\varphi\colon \bX^2\rightarrow[0,1]$.
Given $\eta$, suppose any two distinct points $x,y\in\eta$
are connected with probability $\varphi(x,y)$ independently of all
other pairs. This yields the
RCM, an undirected random graph $\Gamma(\eta)$
with vertex set $\eta$.
A {\em component} ({\em cluster}) of this graph is a maximally connected subset of $\eta$.

We next present some of our results on the stationary (unmarked) Euclidean RCM $\Gamma(\eta)$ (studied in Sections 7--9),
which arises in the case where $\bX=\R^d$ and where $\lambda$ is a positive multiple of the
$d$-dimensional Lebesgue measure $\lambda_d$, i.e., $\eta$ is a stationary Poisson process of intensity $\beta>0$ in $\R^d$. We then assume that $\varphi(x,y)$
depends only on $y-x$ and that $0<\int \varphi(x)\,dx<\infty$,
where $\varphi(x):=\varphi(0,x)$, $x\in\R^d$.
For $k\in\N$ we can label the components of order $k$ ($k$-components)
of $\Gamma(\eta)$ by their lexicographic minima, i.e., their smallest vertices according to the lexicographic order. Let $\eta_k$ denote the
resulting stationary point process.
In the following we denote by $\mathcal{K}^d$ the set
of all compact convex subsets of $\R^d$ with non-empty interior and refer to its elements as convex bodies.
We write $r(W)$ for the inradius of $W\in\mathcal{K}^d$.
We shall consider sequences of convex bodies $(W_n)_{n\in\N}$ such that
$r(W_n)\to\infty$ as $n\to\infty$, and we denote the resulting asymptotic
regime by $r(W)\to\infty$. In the case of random geometric graphs this is basically the same as the thermodynamic limit (see e.g.\ \cite[p.\ 9]{Penrose03}).
Throughout the paper $N$ denotes a standard normal
random variable and $\overset{d}{\longrightarrow}$ stands for convergence in distribution.
We will prove the following
central limit theorem for $\eta_k(W)$, the number of $k$-components of $\Gamma(\eta)$ whose lexicographic minima belong to $W$.

\begin{theorem}\label{t1.1} Let $k\in\N$ and consider
the point process $\eta_k$ of $k$-components in a stationary RCM on $\R^d$ as described above.
Then
\begin{align*}
(\V \eta_k(W))^{-1/2}(\eta_k(W)-\BE \eta_k(W))\overset{d}{\longrightarrow} N
\quad \text{as} \quad r(W)\to\infty.
\end{align*}
\end{theorem}

For an isotropic and monotone connection function Theorem \ref{t1.1} was proved in \cite{BrugMeester04}.
In the case $k\geq2$ the authors assumed the connection function to have bounded support.
For the special case $k=1$, i.e., the number of isolated vertices, a central limit theorem was stated in \cite{RoySarkar2003}, whose proof was erroneous as discussed in \cite{BrugMeester04}. For a different asymptotic regime, Poisson and central limit theorems for component counts in some general RCMs are derived in \cite[Theorem 2.3]{Penrose17}.

Under the stronger assumption that there is a monotonously decreasing function $\tilde\varphi\colon [0,\infty)\to[0,1]$ such that
\begin{equation}\label{eqn:AssumptionTildeVarphi}
\varphi(x)\le \tilde\varphi(|x|),\quad x\in\R^d, \quad \text{and} \quad \int_{\R^d}\tilde\varphi(|x|)^{1/3}\,dx<\infty,
\end{equation}
we shall prove the following quantitative version of
Theorem \ref{t1.1}. We use the Kolmogorov distance $d_K(X,Y)$ between
the distributions of two random variables $X$ and $Y$ (see \eqref{eKolmogorov}), which is the supremum norm of the difference of the distribution functions.

\begin{theorem}\label{t1.2} Let the assumptions of Theorem \ref{t1.1} be satisfied and assume that \eqref{eqn:AssumptionTildeVarphi} holds.
Then there exist constants $C,\tau>0$ only depending on $\beta$, $\varphi$, $\tilde{\varphi}$ and $k$ such that
\begin{equation}\label{eqn:BerryEsseen}
d_K\big((\V \eta_k(W))^{-1/2}(\eta_k(W)-\BE \eta_k(W)),N\big)\le C \lambda_d(W)^{-1/2}
\end{equation}
for all $W\in\mathcal{K}^d$ with $r(W)\ge \tau$.
\end{theorem}

The rate of convergence in \eqref{eqn:BerryEsseen} seems to be quite good since in the classical central limit theorem for the sum of $n$ i.i.d.\ random variables one obtains $1/\sqrt{n}$ and $\lambda_d(W)$ and $n$ play a similar role.

In later results we will not only consider the number of components of the RCM consisting of $k$ vertices, but also count components that are isomorphic to a given finite connected graph.
We will derive multivariate central limit theorems similar to those proved in \cite[Theorem 3.11]{Penrose03} in the special case of random geometric graphs.
In the course of this we will show that the limits
\begin{align*}
\sigma_{\varphi,\varphi}^{(i,j)}:=\lim_{r(W)\to\infty} \frac{\C(\eta_i(W),\eta_j(W))}{\lambda_d(W)},\quad i,j\in\N,
\end{align*}
exist and we will provide explicit formulas for these asymptotic covariances.
Moreover, for all $k\in\N$ the asymptotic variance $\sigma_{\varphi,\varphi}^{(k,k)}$ is positive and for all $m\in\N$ the asymptotic covariance matrix $\big(\sigma_{\varphi,\varphi}^{(i,j)}\big)_{i,j\in[m]}$ is positive definite.

We will also consider the total number of components of the RCM, for which a strong law of large numbers is established in \cite[Theorem 2]{Penrose91}. For $W\in \mathcal{K}^d$
let $\bar\eta(W)$ denote the
number of finite components of $\Gamma(\eta)$ whose vertices
are all in $W$.

\begin{theorem}\label{t1.3} Let the assumptions of Theorem \ref{t1.1}
be satisfied. Then the limit
$\lim_{r(W)\to\infty} (\lambda_d(W))^{-1}\V(\bar\eta(W))$
exists as a positive and finite number and
\begin{align*}
(\V \bar\eta(W))^{-1/2}(\bar\eta(W)-\BE \bar\eta(W))\overset{d}{\longrightarrow} N
\quad \text{as} \quad r(W)\to\infty.
\end{align*}
\end{theorem}

In the case of random geometric graphs a central limit theorem for the total number of components is shown in
Theorem 13.27 in \cite{Penrose03}. Unlike \cite{Penrose03}, our proof does not use
percolation theory. Instead we are using a multivariate version
of Theorem \ref{t1.1} and approximation arguments.

As opposed to the random geometric graph treated in \cite{Penrose03}, in the RCM the edges are drawn randomly, whereby the connection function can inject far reaching dependencies.
As a consequence our proofs require significant new ideas, even though there are several ways to embed a RCM into a marked Poisson process with a sufficiently rich state space.

To prove Theorems \ref{t1.1}-\ref{t1.3} we shall derive
variance inequalities and normal approximation bounds in a much more
general setting. As for the general RCM we let $\eta$ be a
Poisson process on a Borel space $\bX$ with diffuse $\sigma$-finite intensity
measure $\lambda$. Each pair of Poisson points is marked
with a random element taking values in another Borel space $\bM$.
Given $\eta$, these random elements are assumed to be
independent and identically distributed.
We call the resulting point process $\xi$
an {\em edge marking} of $\eta$. For a square integrable
random variable $F$ depending measurably on $\xi$ we introduce difference operators and
derive a variance representation and a \emph{Poincar\'e inequality} as well as quantitative versions
of the central limit theorem from similar results for Poisson functionals.
We achieve this by applying the results from \cite{LaPeSchu16} and \cite{LastPenrose2011} to alternative Poisson representations of edge marked Poisson processes with suitable monotonicity properties and by using some approximation arguments.
We believe that our approach is of independent interest and can be applied to other problems as well. Here, we have, in particular, percolation models with underlying Poisson processes in mind. For example if one independently deletes or colours the edges of a Poisson-Delaunay tessellation, one also obtains an edge marking of a Poisson process. Moreover, we expect that our approach can be generalized to the case where for some $k\in\N$ all $k$-tuples of points of an underlying Poisson process are marked instead of pairs of points.

Our definition of the general RCM includes for $\mathbf{X}=\R^d\times[0,\infty)$ the case of an Euclidean RCM with weights, where one has an underlying marked Poisson process in $\R^d$, i.e., the points are equipped with i.i.d.\ random weights, which are now considered in the connection function as well.
For instance, vertices with larger weights have a higher probability to be connected by edges.
An important example of a RCM with weights is the Boolean model (see e.g.\ \cite{MeesterRoy96}), where balls with i.i.d.\ radii are put around the vertices and two distinct vertices are connected by an edge whenever the corresponding balls intersect.
Another example of the RCM with weights is studied in \cite{DeprezWuethrich16}, as a continuous counterpart to the discrete scale-free percolation model considered in \cite{DeijfenHofstadHooghiemstra}.
It is possible to use our general results to prove
versions of Theorems \ref{t1.1}-\ref{t1.3} for the weighted case, but we confine ourselves to the unweighted case to avoid further technical issues in the proofs. The weighted case will be treated elsewhere.

The paper is organized as follows. In Section \ref{secpre}
we present rigorous definitions of an edge marking and
the general RCM and fix some notation. In Section \ref{secsecond}
we derive formulas for the first and second moments of the
component counts of the general RCM. Section \ref{sanother} gives
another (distributionally equivalent) construction of the edge marking $\xi$ of $\eta$
in terms of an independent marking of the Poisson
process $\eta$. This and a related construction are crucial for proving the results
in Sections \ref{secdiffop} and \ref{sec:GeneralResultNormalApproximation}.
In Section \ref{secdiffop} we consider square integrable random variables $F$ depending measurably on $\xi$
and derive a variance representation in terms of conditional expectations and difference operators as well as a Poincar\'e inequality. In Section \ref{sec:GeneralResultNormalApproximation}
we prove quantitative bounds for the Wasserstein and the Kolmogorov distance
between a standardized $F$ and a standard normal random variable.
In the three final sections we study the RCM with respect to a stationary Poisson process on $\R^d$.
In Section \ref{sasymptoticcov}
we prove the existence and positive definiteness of the asymptotic covariance matrices for component counts,
while Section \ref{seccounts} presents more general versions of
Theorems \ref{t1.1} and \ref{t1.2}. The final section is devoted to the
proof of Theorem \ref{t1.3}.
In the appendix we provide a variance representation for functionals of Poisson processes, which could be of independent interest.

\section{Preliminaries}\label{secpre}

Let $(\bX,\cX)$ be a Borel space, i.e., $(\bX,\cX)$ is a measurable space and there is a Borel measurable bijection $T$ from $\bX$ to a Borel subset of $(0,1]$ with measurable inverse. By $\lambda$ we denote a diffuse and $\sigma$-finite measure on $\bX$.
We assume that $\bX$ is equipped with
a transitive binary relation $\prec$ such that
$\{(y,z):y\prec z\}$ is a measurable subset of $\bX^2$
and such that $x\prec x$ fails for all $x\in\bX$.
We also require
that $\lambda([x])=0$, $x\in\bX$,
where $[x]:=\bX\setminus \{z\in\bX:\text{$z\prec x$ or $x\prec z$}\}$.
Note that $x\in[x]$ for all $x\in\bX$.
The diffuseness of $\lambda$ and the existence of the
partial order $\prec$ are no restrictions of generality.
Indeed, let $(\bX',\cX')$ be an arbitrary Borel space equipped
with a $\sigma$-finite measure $\lambda'$.
Then we can extend the underlying space to $\bX:=\bX'\times[0,1]$ and
$\lambda:=\lambda'\otimes \lambda_1|_{[0,1]}$,
where $\lambda_1|_{[0,1]}$ is the restriction of the Lebesgue measure on $\R$ to the unit interval $[0,1]$,
and define the order relation on $\bX$ by $(x',s)\prec(y',t)$ if $s<t$.

All random objects ocurring in this paper are defined
over a fixed probability space $(\Omega,\mathcal{F},\BP)$.
A {\em point process} (on $\mathbf{X}$) is a random element $\eta$ of the space
$\bN(\bX)$ of all at most countably infinite subsets $\mu$ of  $\bX$, equipped with
the smallest $\sigma$-field $\mathcal{N}(\bX)$ making the mappings
$\mu\mapsto \mu(B):=\card(\mu \cap B)$ measurable for all $B\in\mathcal{X}$.
Then $\eta(B)$ (defined as the mapping $\omega\mapsto\eta(\omega,B):=\eta(\omega)(B)$)
is a random variable for each $B\in\mathcal{X}$. In fact,
$\eta$ is a {\em simple} point process, i.e., $\eta$ can be interpreted as a random counting measure without multiplicities; see
e.g.\ \cite[Chapter 6]{LastPenrose}. The {\em intensity measure} of
a point process $\eta$ is the measure $\cX\ni B\mapsto \mathbb{E}\eta(B)$ on $\mathbf{X}$.

Throughout the paper $\eta$ will denote
a Poisson (point) process with
intensity measure $\lambda$, which is defined by the following two properties (see e.g.\ Definition 3.1 and
Proposition 6.9 in \cite{LastPenrose} or \cite{Kallenberg17}):
\begin{enumerate}[(i)]
\item For every $B\in\mathcal{X}$ the distribution of $\eta(B)$ is Poisson with parameter $\lambda(B)$.
\item For every $n\in\N$ and pairwise disjoint sets $B_1,\ldots,B_n\in\mathcal{X}$ the random variables $\eta(B_1),\ldots,\eta(B_n)$ are independent.
\end{enumerate}
For $k \in \N_0\cup\{\infty\}$ we set $[k] := \{n\in\N : n \le k\}$ and $[k]_0 := [k] \cup\{0\}$.
Note that $[k] = \emptyset$ if $k=0$ and $[k] = \N$ if $k = \infty$.
Since $\bX$ is a Borel space and $\lambda$ is $\sigma$-finite,
\cite[Corollary 6.5]{LastPenrose} shows that
$\eta$ can be almost surely represented as
$$
  \eta=\{X_n:n\in [\kappa]\},
$$
where the $X_n$, $n\in\N$, are random elements of $\bX$, $\kappa:=\eta(\bX)$ is
a random element of $\N_0\cup\{\infty\}$, and
the  $X_n$ are measurable functions of  $\eta$.

The space $\bX^{[2]}:=\{e\in\bN(\bX):e(\bX)=2\}$ is a measurable subset
of $\bN$. When restricting the $\sigma$-field on $\mathbf{N}(\bX)$
to subsets of $\bX^{[2]}$,
this space becomes a Borel space; see \cite[Lemma 1.7]{Kallenberg17}.
Later, any $e\in \bX^{[2]}$ is a potential edge of the RCM.
Let $(\bM,\mathcal{M})$ be a further Borel space
and let $Z_{m,n}$, $m,n\in\N$, be independent random elements on $\bM$
with common distribution $\BM$. Assume that
the double sequence $(Z_{m,n})_{m,n\in\N}$ is independent of $\eta$.
Then
\begin{align}\label{e2.3}
\xi:=\{(\{X_m,X_n\},Z_{m,n}):X_m\prec X_n,\, m,n\in[\kappa]\}
\end{align}
is a point process on $\bX^{[2]}\times\bM$, namely an
({\em independent}) {\em edge marking of $\eta$}.
Note that
the Poisson process $\eta$ can be reconstructed from $\xi$.

Let $\varphi\colon\mathbf{X}^2\to[0,1]$ be a measurable and symmetric connection function.
To define the RCM we use $\xi$ in
the case $\bM=[0,1]$ and $\BM=\lambda_1|_{[0,1]}$.
Then
\begin{align}\label{e2.5}
\chi:=\{\{X_m,X_n\}:X_m\prec X_n,\, Z_{m,n}\le \varphi(X_m,X_n),\, m,n\in[\kappa]\}
\end{align}
is a point process on $\bX^{[2]}$, namely the set of edges.
The pair $\Gamma(\eta):=(\eta,\chi)$ is our RCM.
If several connection functions occur, we indicate the dependence on the connection function by writing $\Gamma_\varphi(\eta)$ instead of $\Gamma(\eta)$.
Note that the random graph $\Gamma(\eta)$ does not only depend on $\eta$ but on the whole
point process $\xi$. In the same way we define
$\Gamma(\zeta)=\Gamma_\varphi(\zeta)$ for any simple point process $\zeta$ on $\bX$.

A {\em component} of $\Gamma(\eta)$ is a set $\mu\subset\eta$  such
that the graph with vertex set $\mu$
and edges induced by $\Gamma(\eta)$
is connected and, moreover, no point of $\mu$ is connected to a point
of $\eta\setminus\mu$.

For any graph $G=(V,E)$ let $|G|:=\text{card}(V)$ denote the \emph{order} of $G$.
Two graphs $G=(V,E)$ and $G'=(V',E')$ are {\em isomorphic}
if there is a bijection $T\colon V\rightarrow V'$ such
that $\{x,y\}\in E$ if and only if $\{T(x),T(y)\}\in E'$
for all $x,y\in V$ with $x\ne y$. In this case we write
$G\simeq G'$.
For $k\in\N$ let $\bG_k$ be a set of connected graphs with vertex set
$[k]$, containing exactly one member of each equivalence
class. Then $\bG:=\bigcup_{k=1}^\infty \bG_k$ is up to isomorphy the
set of all finite connected graphs.

We denote by $\mathbf{G}_{k,\varphi}$ and
$\mathbf{G}_\varphi$ the graphs of $\mathbf{G}_{k}$ and $\mathbf{G}$,
respectively, which occur as components in the RCM $\Gamma_\varphi(\eta)$ with
positive probability. Note that these can be strict subsets
depending on the choice of $\varphi$. For example if
$\varphi(x,y)={\bf 1}\{|x-y|\leq r\}$, $x,y\in\R^2$, for some fixed $r>0$,
the resulting RCM in the plane, which is the random geometric graph, has almost surely
no components that are isomorphic to the graph
$([7],\{\{1,i\}: i\in\{2,\hdots,7\}\})$.
For $m\in\N$ we let
\begin{equation}\label{eqn:set.G.ne}
\mathbf{G}_\varphi^{m,\ne} := \big\{(G_1,\ldots,G_m) \in \mathbf{G}_\varphi^m : G_1,\ldots,G_m \text{ are distinct}\big\}.
\end{equation}

For $k\in\N$ and $G\in\bG_k$ let $\eta_{G}$ denote the point process
\begin{align*}
\eta_G :=\eta_{\varphi,G} := &\big\{ x_1 \in \eta : \exists (x_2,\ldots,x_{k}) \in \eta^{k-1} \text{ with } x_1\prec x_2\prec\cdots\prec x_k \text{ and} \\
&\quad \{x_1,x_2,\ldots,x_k\} \text{ is a component of } \Gamma_\varphi(\eta) \text{ isomorphic to } G \big\},
\end{align*}
i.e., the point process of lexicographic minima of components of the RCM isomorphic to $G$. Similarly, for each $k\in\N$ let
$$
\eta_{k}:=\eta_{\varphi,k}:=\bigcup_{G\in\bG_k}\eta_{\varphi,G}
$$
be the point process of lexicographic minima of $k$-components of the RCM.

We finish this section with some notation that is used in the sequel. Since we will often consider the probability that two vertices of our RCM are not connected by an edge, we introduce the abbreviation $\bar\varphi:=1-\varphi$.

For $x\in\R^d$ and $r>0$ we denote by $B^d(x,r)$ the closed $d$-dimensional ball with centre $x$ and radius $r$. For $x\in\R^d$ and a compact set $A \subset \R^d$ we define $d(x,A):=\min_{y\in A}|x-y|$.

For a graph $G$, vertices $x_1$ and $x_2$ of $G$ and $m\in\N$ we mean by $x_1 \overset{\le m}{\llra}x_2 \textit{ in }G$ that $x_1$ and $x_2$ are connected via a path in $G$ with at most $m$ edges.
If $m=1$, we write $x_1\leftrightarrow x_2 \textit{ in }G$.
Similarly, we write $x_1\not\lra x_2 \textit{ in } G$ if $x$ and $y$ are not connected by an edge in $G$.
For some set $W$ and a vertex $x$ of $G$ let $x\overset{\le m}{\llra}W \textit{ in }G$ denote the event that there exists a vertex $y$ of $G$ such that $y \in W$ and $x\overset{\le m}{\llra}y \textit{ in } G$, or $x\in W$.
For a vertex $x$ of $G$ let $\deg(x,G)$ be the degree of $x$ in $G$, i.e., the number of edges incident to $x$ in $G$.

\section{First and second order properties}\label{secsecond}

In this section we consider the RCM $\Gamma(\eta)$ introduced
in Section \ref{secpre}.
We shall study first and second moment properties of the component counts.
The next result shows that for $G\in\bG_k$ with $k\in\N$ the intensity measure of $\eta_{G}$ can be expressed in terms of the function
\begin{align}\label{e3.2}
p_{\varphi,G}(x_1,\ldots,x_k)
:=\I\{x_1\prec\cdots \prec x_k\}\BP(\Gamma_\varphi(\{x_1,\ldots,x_k\})\simeq G),
\quad x_1,\ldots,x_k\in\bX.
\end{align}
 Recall that $\bar\varphi:=1-\varphi$.

\begin{proposition}\label{p1.2} Let $k\in\N$, $G\in\bG_k$ and $W\in\cX$. Then
\begin{align}\label{intcluster}
\BE&\eta_G(W)\\ \notag
&=\int\I\{x_1\in W\}
p_{\varphi,G}(x_1,\ldots,x_k)
\exp\Bigg[\int \Bigg(\prod^k_{i=1}\bar\varphi(x_i,y)-1\Bigg)\,\lambda(dy)\Bigg]
\,\lambda^k(d(x_1,\dots,x_k)).
\end{align}
\end{proposition}

For the Euclidean case similar formulas as \eqref{intcluster} for the expected number of $k$-components can be found in \cite[Proposition 1]{Penrose91} as well as for $k=1$ in \cite[Lemma 4]{RoySarkar2003} and \cite[Lemma 3.2]{BrugMeester04}, while the asymptotic expectations in the special case of random geometric graphs are considered in \cite[Proposition 3.3]{Penrose03}.

\begin{proof}[Proof of Proposition \ref{p1.2}]
For $k$ distinct points $x_1,\dots,x_k\in\eta$  let
$G(x_1,\ldots,x_k,\Gamma(\eta))$ denote the graph with vertices
$x_1,\ldots,x_k$ and edges induced by $\Gamma(\eta)$.
This graph is a component isomorphic to $G$ if and
only if $G(x_1,\dotsc,x_k,\eta)\simeq G$ and none
of the $x_i$ is connected to a point in
$\eta\setminus\{x_1,\ldots,x_k\}$.
Given $\eta$, these two events are independent and have
probabilities
$$
\BP(\Gamma(\{x_1,\ldots,x_k\})\simeq G) \qquad \text{and} \qquad\prod_{y\in\eta\setminus\{x_1,\ldots,x_k\}}\prod^k_{i=1}\bar\varphi(x_i,y),
$$
respectively. Using the multivariate Mecke equation (see e.g.\ \cite[Theorem 4.4]{LastPenrose})
it follows that
\begin{align*}
\BE\eta_G(W)=
&\int\I\{x_1\in W\}\, \1\{x_1\prec\cdots\prec x_k\}\,\BP(\Gamma(\{x_1,\ldots,x_k\})\simeq G)\\
&\times\BE\Bigg[\prod_{y\in\eta}\prod^k_{i=1}\bar\varphi(x_i,y)\Bigg]
\,\lambda^k(d(x_1,\dots,x_k)).
\end{align*}
The formula for the generating functional of a Poisson
process (see \cite[Exercise 3.6]{LastPenrose}) implies the result.
\end{proof}

Together with $\Gamma_\varphi(\eta)$ we consider a second RCM
$\Gamma_\psi(\eta)$ based on another connection function $\psi: \mathbf{X}^2\to[0,1]$.
The edges of $\Gamma_\psi(\eta)$ are defined by \eqref{e2.5} with
$\psi$ in place of $\varphi$. As a result, the two
random graphs $\Gamma_\varphi(\eta)$ and $\Gamma_\psi(\eta)$ are strongly
coupled. If, for instance, $\psi\le\varphi$, then each edge
of $\Gamma_\psi(\eta)$ is also an edge of $\Gamma_\varphi(\eta)$.

\begin{proposition}\label{p3.7} Let $k,l\in\N$,
$G\in\bG_k$ and $H\in\bG_l$. Let $\varphi,\psi: \mathbf{X}^2\to[0,1]$ be two connection
functions with $\psi\le \varphi$ and let $W,W'\in\cX$. Then
\begin{align*}
&\BE\eta_{\varphi,G}(W)\eta_{\psi,H}(W')=\int\I\{x_1\in W,x_{k+1}\in W'\}p_{\varphi,G}(x_1,\ldots,x_k)
p_{\psi,H}(x_{k+1},\ldots,x_{k+l})\\
&\times \prod^k_{i=1}\prod^{k+l}_{j=k+1}\bar\varphi(x_i,x_j)
\exp\Bigg[\int \Bigg(\prod^k_{i=1}\bar\varphi(x_i,y)\prod^{k+l}_{j=k+1}\bar\psi(x_j,y)-1\Bigg)\,
\lambda(dy)\Bigg]
\,\lambda^{k+l}(d(x_1,\dots,x_{k+l}))\\
&+\I\{k=l\}\int\I\{x_1\in W\cap W'\} \BP(\Gamma_\varphi(\{x_1,\ldots,x_k\})\simeq G, \Gamma_\psi(\{x_1,\ldots,x_k\})\simeq H)\\
&\qquad\qquad\qquad \times \1\{x_1\prec\cdots\prec x_k\} \exp\Bigg[\int \Bigg(\prod^k_{i=1}\bar\varphi(x_i,y)-1\Bigg)\,
\lambda(dy)\Bigg]\,\lambda^k(d(x_1,\dots,x_k)).
\end{align*}
\end{proposition}

\begin{proof}
Let $x_1,\dots,x_{k+l}\in\eta$ be such that $x_1,\hdots,x_k$ are distinct and $x_{k+1},\hdots,x_{k+l}$ are distinct.
Then $G(x_1,\ldots,x_k,\Gamma_\varphi(\eta))$
is a component isomorphic to $G$ and
$G(x_{k+1},\ldots,{x_{k+l}},\Gamma_\psi(\eta))$
is a component isomorphic to $H$ if and only if one of the following two cases occurs.
In the first case we have
$\{x_1,\dots,x_{k}\}\cap\{x_{k+1},\dots,x_{k+l}\}=\emptyset$,
$G(x_1,\ldots,x_k,\Gamma_\varphi(\eta))\simeq G$,
$G(x_{k+1},\ldots,{x_{k+l}},\Gamma_\psi(\eta))\simeq H$,
no point from $\{x_1,\dots,x_{k}\}$ is connected with a point
from $\{x_{k+1},\dots,x_{k+l}\}$ via edges in $\Gamma_\varphi(\eta)$,
no point from $\{x_1,\dots,x_{k}\}$ is connected to $\eta\setminus\{x_1,\ldots,x_{k+l}\}$
via edges in $\Gamma_\varphi(\eta)$, and
no point from $\{x_{k+1},\dots,x_{k+l}\}$ is connected to $\eta\setminus\{x_1,\ldots,x_{k+l}\}$
via edges in $\Gamma_\psi(\eta)$. Given $\eta$, the conditional probability of this event
equals
\begin{align*}
&\BP(\Gamma_\varphi(\{x_1,\ldots,x_k\})\simeq G)
\BP(\Gamma_\psi(\{x_{k+1},\ldots,x_{k+l}\})\simeq H)\\
&\quad\times \prod^k_{i=1}\prod^{k+l}_{j=k+1}\bar\varphi(x_i,x_j)\prod_{y\in\eta\setminus\{x_1,\ldots,x_{k+l}\}}
\prod^k_{i=1}\bar\varphi(x_i,y)\prod^{k+l}_{j=k+1}\bar\psi(x_j,y).
\end{align*}
In the second case we have that $k=l$,
$\{x_1,\dots,x_{k}\}=\{x_{k+1},\dots,x_{2k}\}$,
$G(x_1,\ldots,x_k,\Gamma_\varphi(\eta))\simeq G$, $G(x_1,\ldots,x_k,\Gamma_\psi(\eta))\simeq H$  and none
of $x_1,\hdots,x_k$ is connected to a point in
$\eta\setminus\{x_1,\ldots,x_k\}$. The conditional probability of this event equals
$$
\BP(\Gamma_\varphi(\{x_1,\ldots,x_k\})\simeq G, \Gamma_\psi(\{x_1,\ldots,x_k\})\simeq H)
\prod_{y\in\eta\setminus\{x_1,\ldots,x_k\}}\prod^k_{i=1}\bar\varphi(x_i,y).
$$
The multivariate Mecke equation and the formula for the generating functional of a Poisson
process (combined with a symmetry argument) yield the result similarly as in the proof of Proposition \ref{p1.2}.
\end{proof}

\section{Another description of the RCM and difference operators}\label{sanother}

For our later purposes it is useful to define a version of $\xi$ (see \eqref{e2.3})
in terms of an independent marking $\eta^*$ of $\eta$ and to use
that $\eta^*$ is a Poisson process. As mark space
we take $\bY:= \bM^{\N\times\N}$, that is the set
of all double sequences $(u_{m,n})^\infty_{m,n=1}$ with values in $\bM$.
We define $\BQ:=\BM^{\N\times\N}$ so that $\BQ$ is the
distribution of a double sequence of independent random elements
with distribution $\BM$. Let $\eta^*$ be an independent
$\BQ$-marking of $\eta$ and note that $\eta^*$ is a Poisson process
on $\bX\times\bY$ with intensity measure $\lambda\otimes\BQ$;
see e.g.\ \cite[Theorem 5.6]{LastPenrose}.

We assert that for any $\varepsilon\in(0,1)$ there is a measurable partition $\{B^\varepsilon_k:k\in\N\}$
of the space $\bX$ such that $\lambda(B^\varepsilon_k)\le\varepsilon$
for each $k\in\N$. For a fixed $\varepsilon\in(0,1)$ this can be shown as follows. By the Borel property
of $\bX$ we may assume that $\bX$ is a Borel subset of the
interval $(0,1]$. Since $\lambda$ is $\sigma$-finite, there
exists a measurable partition $\{C_k:k\in\N\}$ of
$\bX$ such that $\lambda(C_k)<\infty$ for each $k\in\N$.
Now the assertion follows from the observation that one can choose for any Borel set $A\subset (0,1]$ with $\lambda(A)<\infty$ two disjoint Borel sets $A_1,A_2\subset(0,1]$ with $A_1\cup A_2=A$ and $\lambda(A_1)=\lambda(A_2)=\lambda(A)/2$. Indeed, the measure $\nu(B):=\lambda(A\cap B)$ for Borel sets $B\subset (0,1]$ is diffuse and finite so that there exists a $t\in(0,1]$ such that $\nu((0,t])=\nu((t,1])= \nu((0,1])/2=\lambda(A)/2$. Thus, the Borel sets $A_1:=A\cap (0,t]$ and $A_2:=A\cap (t,1]$ have the desired properties.
For $x \in \bX$ let $B^\varepsilon(x)$ be the unique element of the
partition $\{B^\varepsilon_k:k\in\N\}$
containing $x$. We can assume that
\begin{equation}\label{eqn:MonotonicityPartition}
B^{\varepsilon_1}(x)\subset B^{\varepsilon_2}(x)
\end{equation}
for all $x\in \bX$ and all $0<\varepsilon_1\le \varepsilon_2<1$. This can be achieved as follows.
First we refine the partitions such that
$B^{1/(n+1)}(x)\subset B^{1/n}(x)$ holds for all $x\in \bX$ and all
$n\in\N$.
Then we define $B^{\varepsilon}_k:=B^{1/(n+1)}_k$ for
$1/(n+1)\le\varepsilon< 1/n$
and $k\in\N$.

Take $(x,u)\in\eta^*$, where $u=(u_{m,n})_{m,n=1}^\infty$. Let $k\in\N$ and set
$$
r:=\card\big\{x'\in \eta\cap B^\varepsilon_k:x'\prec x\big\}.
$$
Let $x_1,\ldots,x_r\in\eta\cap B^\varepsilon_k$ be such that
$x_1\prec\cdots \prec x_r\prec x$.
For each $i\in[r]$
we define $U_\varepsilon(\eta^*,x,x_i):=u_{k,i}$.
Almost surely we have that for all distinct $x,x'\in\eta$ either
$x'\prec x$ or $x\prec x'$. In the first case $U_\varepsilon(\eta^*,x,x')$
and in the second $U_\varepsilon(\eta^*,x',x)$ is well defined by the above procedure.
For all other $x,x'\in\bX$ we let $U_\varepsilon(\eta^*,x',x):=u_0$
for some fixed $u_0\in \bM$.
Then
$$
\xi^*_\varepsilon:=\big\{\big(\{x,x'\},U_\varepsilon(\eta^*,x,x')\big):(x,x')\in\eta^2,\, x'\prec x \big\}
$$
is a point process on $\bX^{[2]}\times\bM$ satisfying $\xi^*_\varepsilon\overset{d}{=}\xi$, where $\overset{d}{=}$ denotes equality in distribution.
In fact we have $\xi^*_\varepsilon=T_\varepsilon(\eta^*)$ for a well-defined measurable mapping $T_\varepsilon\colon \bN(\bX\times\bY)\to \bN(\bX^{[2]}\times\bM)$, whence
\begin{equation} \label{eqn:XiEtaStar}
\xi \overset{d}{=} T_\varepsilon(\eta^*).
\end{equation}

Consider the edge marking $\xi$ of the Poisson process
$\eta$, defined by \eqref{e2.3}.
Difference operators play a fundamental role in the stochastic analysis of Poisson functionals, see e.g.\ \cite{LaPeSchu16,LastPenrose,PeccatiReitzner}.
In the following we generalize these operators, so as to apply
to functions of the point process $\xi$.
Let $L_\xi$ denote the space of all $\sigma(\xi)$-measurable
random elements of $\R$. For each $F\in L_\xi$ there is a measurable $f\colon \bN(\bX^{[2]}\times\bM)\to\R$
such that $F=f(\xi)$ almost surely. We call $f$ a {\em representative} of $F$.
Our results will not depend on the choice of $f$.

We extend the (double) sequence
$(Z_{m,n})_{m,n=1}^\infty$ featuring in \eqref{e2.3} to a sequence
$(Z_{m,n})_{m,n \in \Z}$ of independent $\bM$-valued random elements
with distribution $\BM$, independent of the Poisson process $\eta$.
Let $k\in\N$, $x_1,\ldots, x_k \in \bX$ and $I\subset [k]$.
We define a point process $\xi_{(x_i)_{i \in I}}$ on $\bX^{[2]}\times\bM$ by
\begin{align*}
\xi_{(x_i)_{i \in I}}
:=\{(\{X_m,X_n \}, Z_{m,n}):X_m \prec X_n, \,m,n \in[\kappa] \cup \{ -i : i \in I \}\}
\end{align*}
where $X_{-i} := x_i$, $i\in I$, and $X_i$, $i\in[\kappa]$, are the points of $\eta$. Note that $\xi_{(x_i)_{i \in I}}=\xi$
if $I=\emptyset$. In the case $I=[k]$ we write
$\xi_{x_1,\ldots,x_k}$ instead of $\xi_{(x_i)_{i \in I}}$.
Given $x_1,\ldots,x_k \in \bX$ and $I\subset[k]$, the point process $\xi_{(x_i)_{i \in I}}$
is a measurable function of $\xi_{x_1,\ldots,x_k}$.
The joint distribution of $(\eta, \xi_{x_1,\ldots,x_k})$ is a measurable function of $x_1,\ldots,x_k \in \mathbf{X}$ that is invariant under permutations.

The  {\em multivariate Mecke equation} for Poisson processes (see e.g.\ \cite[Theorem 4.4]{LastPenrose})
can easily be extended to measurable functions
$g\colon \bX^k\times\bN(\bX^{[2]}\times\bM)\to [0,\infty)$.
We have
\begin{align}\label{Mecke}
\BE \sideset{}{^{\ne}}\sum_{(x_1,\ldots,x_k)\in\eta^k} g(x_1,\ldots,x_k,\xi)
=\int \BE g(x_1,\ldots,x_k,\xi_{x_1,\ldots,x_k})\,\lambda^k(d(x_1,\ldots,x_k)),
\end{align}
where the sum on the left extends over all $(x_1,\ldots,x_k)\in\eta^k$
such that $x_i\ne x_j$ for $i\ne j$.

Let $F\in L_\xi$ have representative $f$.
For each $k\in\N$ and all $x_1,\ldots, x_k \in \bX$ we
define a random variable $\Delta^k_{x_1,\ldots,x_k}F$ by
\begin{align}\label{diffop}
\Delta^k_{x_1,\ldots,x_k}F:= \sum_{I\subset [k]} (-1)^{k -|I|} f\big(\xi_{(x_i)_{i \in I}}\big).
\end{align}
In particular, this gives us
$$
\Delta_{x_1} F := \Delta^1_{x_1} F :=f(\xi_{x_1})-f(\xi) \quad \text{and} \quad
\Delta^2_{x_1,x_2} F:=f(\xi_{x_1,x_2})-f(\xi_{x_1})-f(\xi_{x_2})+f(\xi).
$$
The definition \eqref{diffop} is justified since for two representatives $f_1$ and $f_2$ of $F$ by the Mecke equation \eqref{Mecke} for any $m\in\N$, $f_1(\xi_{x_1,\hdots,x_m})=f_2(\xi_{x_1,\hdots,x_m})$ $\P$-a.s.\ for $\lambda^{m}$-a.e.\ $(x_1,\ldots,x_m)\in\bX^m$.

For $\varepsilon>0$ and $F\in L_\xi$ with representative $f$ let $F^*_\varepsilon=f\big(T_\varepsilon(\eta^*)\big)$, see \eqref{eqn:XiEtaStar}. For $k\in\N$, $x_1,\hdots,x_k\in\bX$ and $y_1,\hdots,y_k\in\bY$ define
$$
F^*_{\varepsilon,(x_i,y_i)_{i\in I}}:=f\circ T_\varepsilon\big(\eta^*\cup\{(x_i,y_i):i\in I\}\big)
$$
for $I\subset [k]$, which equals $F^*_\varepsilon$ in the case $I=\emptyset$. Further let
$$
D^k_{(x_1,y_1),\hdots,(x_k,y_k)}F^*_\varepsilon:=\sum_{I\subset [k]} (-1)^{k -|I|}F^*_{\varepsilon,(x_i,y_i)_{i\in I}}.
$$
Note that this is the usual iterated difference operator for functions of the Poisson process $\eta^*$, see e.g.\ \cite[Equation (18.3)]{LastPenrose}.
In the following we link the difference operators $D$ and $\Delta$ in order to transfer results for the Poisson process $\eta^*$ to the edge marked Poisson process $\xi$.

From now on let $(Y_n)_{n\in\N}$ be independent random elements of $\bY$ with distribution $\mathbb{Q}$ which are independent from everything else. For $\varepsilon>0$ and $k\in\N$ assume that $x_1\in B^\varepsilon_{n_1},\hdots, x_k\in B^\varepsilon_{n_k}$ with distinct $n_1,\hdots,n_k\in\N$.
Then we obtain from the preceding construction that
\begin{equation}\label{eqn:DistributionalIdentity}
\1\big\{\eta\big(B_{n_i}^\varepsilon\big)=0, i\in[k]\big\} \big(1, (F^*_{\varepsilon, (x_i,Y_i)_{i\in I}})_{I\subset [k]} \big) \overset{d}{=} \1\big\{\eta\big(B_{n_i}^\varepsilon\big)=0, i\in[k]\big\} \big(1, (f\big(\xi_{(x_i)_{i\in I}}\big))_{I\subset [k]} \big).
\end{equation}
Assume now that $f$ is bounded, that is $\|f\|_\infty:=\sup\{|f(\mu)|: \mu\in\mathbf{N}(\bX^{[2]}\times\bM)\}<\infty$, and let $g: \R^{2^k}\to\R$ be measurable and bounded. Then \eqref{eqn:DistributionalIdentity} implies that
\begin{align} \label{eqn:IdentityDDelta} \notag
\BE \1\big\{\eta\big(B^\varepsilon(x_i)\big)=0, &\, i\in[k]\big\} g\Big(\big(D^{|I|}_{(x_i,Y_i)_{i\in I}}F^*_\varepsilon\big)_{I\subset[k]}\Big) \\
& = \BE \1\big\{\eta\big(B^\varepsilon(x_i)\big)=0, i\in[k]\big\}g\Big(\big(\Delta^{|I|}_{(x_i)_{i\in I}}F\big)_{I\subset[k]}\Big) .
\end{align}
Note that \eqref{eqn:DistributionalIdentity} and \eqref{eqn:IdentityDDelta} are both subject to the condition that $B^\varepsilon(x_1),\hdots,B^\varepsilon(x_k)$ are distinct. But because of $\lambda(B^{\varepsilon}(x))\to 0$ as $\varepsilon\to 0$ and the monotonicity property \eqref{eqn:MonotonicityPartition}, we have that
\begin{equation}\label{eqn:x1..xk}
\lim_{\varepsilon\to 0} \1\big\{B^\varepsilon(x_1),\hdots,B^\varepsilon(x_k) \text{ are distinct}\big\}=1
\end{equation}
for $\lambda^k$-a.e.\ $(x_1,\hdots,x_k)\in\bX^k$.

\section{Variance formulas}\label{secdiffop}

In order to deduce an exact variance representation for a function of an edge-marked Poisson process, we use a further construction to obtain an edge marking of a Poisson process. In the following let $\hat{\eta}$ be a Poisson process on $\mathbf{X}\times [0,1]\times \bM^{\N\times\N}$ with intensity measure $\lambda \otimes \lambda_1|_{[0,1]}\otimes \mathbb{Q}$.
For a point $\hat{x}\in\hat{\eta}$ we interpret the first component as a location in $\mathbf{X}$, the second component as birth time and the third component as a double sequence of marks.
For $\mu\in\mathbf{N}\big(\mathbf{X}\times [0,1]\times \bM^{\N\times\N}\big)$ and $s,t\in[0,1]$ with $s<t$ we denote by $\mu_{[s,t)}$ the restriction of $\mu$ to $\mathbf{X}\times[s,t)\times \bM^{\N\times\N}$. For $t\in[0,1]$ we write $\mu_t:=\mu_{[0,t)}$ and $\BE[\cdot | \hat{\eta}_t]$ stands for the conditional expectation with respect to the sigma-field generated by $\hat{\eta}_t$.

For some $\varepsilon\in(0,1)$ let $\{B_k^{\varepsilon}: k\in\N\}$ be a measurable partition of $\bX$ such that $\lambda(B_k^{\varepsilon})\leq \varepsilon$ for $k\in\mathbb{N}$ as in Section \ref{sanother}. From $\hat{\eta}$ we can derive an independent edge marking in the following way. For $\big(x_1,t_1,(u^{(1)}_{i,j})\big), \big(x_2,t_2,(u^{(2)}_{i,j})\big) \in\hat{\eta}$ with $0\leq t_1<t_2\leq 1$, we mark the edge $\{x_1,x_2\}$ according to the following rule. If $x_1\in B^{\varepsilon}_n$, we order all points of $\hat{\eta}$ in $B^{\varepsilon}_n$ according to their birth times. Assume that $x_1$ is the $m$-th oldest of these points. Then we mark the edge $\{x_1,x_2\}$ with $u^{(2)}_{n,m}$. Formally, we can think of this construction as a measurable map
$T: \mathbf{N}\big(\bX \times [0,1]\times \bM^{\N\times\N}\big)\to \mathbf{N}\big(\bX^{[2]} \times \bM\big)$.
For $\mu\in\mathbf{N}\big(\bX\times [0,1]\times \bM^{\N\times\N}\big)$ and $(x,t,M)\in \mu$ we define $T(\mu)\setminus\{x\}$ as $T(\mu)$ without the point $x$ and all corresponding edge marks.

\begin{theorem}\label{thm:VarianceRepresentation}
Let $F=f(T(\hat{\eta}))$ with $f:\bN(\bX^{[2]}\times\bM)\to \R$ measurable and $\BE F^2<\infty$. Then
$$
\V F = \int \int_0^1  \int \BE\big[ \BE\big[ f(T(\hat{\eta}\cup \{(x,t,M)\})) - f(T(\hat{\eta}\cup \{(x,t,M)\})\setminus\{x\}) | \hat{\eta}_t \big]^2 \big] \, \mathbb{Q}(dM) \,\, dt \,  \lambda(dx).
$$
\end{theorem}

\begin{proof}
It follows from Theorem \ref{thm:appendix.variance.repre} in the appendix that
$$
\V F = \int \int_0^1 \int \BE\big[ \BE\big[ f(T(\hat{\eta}\cup \{(x,t,M)\})) - f(T(\hat{\eta})) | \hat{\eta}_t \big]^2 \big] \, \mathbb{Q}(dM) \,  dt \, \lambda(dx).
$$
For $t\in[0,1]$ let $\mathbb{P}_{[t,1)}$ be the distribution of $\hat{\eta}_{[t,1)}$. For $(x,t,M)\in \mathbf{X}\times [0,1]\times \bM^{\N\times\N}$, $\zeta$ distributed according to $\mathbb{P}_{[t,1)}$ and $\mu\in \mathbf{N}\big(\bX \times [0,t)\times \bM^{\N\times\N}\big)$ it follows from the above construction of the edge marking that $T(\mu\cup\zeta)$ and $T(\mu\cup\zeta\cup\{(x,t,M)\})\setminus\{x\}$ have the same distribution. Thus, we obtain that
\begin{align*}
& \BE\big[ f(T(\hat{\eta}\cup \{(x,t,M)\})) - f(T(\hat{\eta})) | \hat{\eta}_t \big] \\
& = \int f\big(T\big(\hat{\eta}_t\cup\zeta\cup \{(x,t,M)\}\big)\big) - f\big(T\big(\hat{\eta}_t\cup\zeta\big)\big) \, \mathbb{P}_{[t,1)}(d\zeta) \\
& = \int f\big(T\big(\hat{\eta}_t\cup\zeta\cup \{(x,t,M)\}\big)\big) - f\big(T\big(\hat{\eta}_t\cup\zeta\cup \{(x,t,M)\}\big)\setminus\{x\}\big) \, \mathbb{P}_{[t,1)}(d\zeta) \\
& = \BE\big[ f(T(\hat{\eta}\cup \{(x,t,M)\})) - f(T(\hat{\eta}\cup \{(x,t,M)\})\setminus\{x\}) | \hat{\eta}_t \big],
\end{align*}
which completes the proof.
\end{proof}

From Theorem \ref{thm:VarianceRepresentation} we can deduce the following {\em Poincar\'e inequality}
for square integrable functionals of $\xi$. For the Poincar\'e inequality for Poisson functionals see e.g.\ \cite[Theorem 18.7]{LastPenrose}.

\begin{theorem}\label{tpoincare} Let $F\in L_\xi$ satisfy
$\BE F^2<\infty$. Then
\begin{align}\label{epoincare}
\V F &\le \int \BE (\Delta_xF)^2 \, \lambda(dx).
\end{align}
\end{theorem}

\begin{proof}
Let $f$ be a representative of $F$. Using the Jensen inequality and the fact that $f(T(\hat{\eta}\cup \{(x,t,M)\})) - f(T(\hat{\eta}\cup \{(x,t,M)\})\setminus\{x\})$ has the same distribution as $\Delta_x F$, we obtain from Theorem \ref{thm:VarianceRepresentation} that
\begin{align*}
\V F & \leq \int \int_0^1 \int \BE ( f(T(\hat{\eta}\cup \{(x,t,M)\})) - f(T(\hat{\eta}\cup \{(x,t,M)\})\setminus\{x\}) )^2 \, \Lambda(dM) \, dt \, \lambda(dx) \\
&  = \int\BE ( \Delta_xF )^2  \, \lambda(dx),
\end{align*}
which is the desired inequality.
\end{proof}

\section{Normal approximation}\label{sec:GeneralResultNormalApproximation}\label{sec:normal.approx}

As before let $\xi$ be the edge marking  of the Poisson process
$\eta$,  defined by \eqref{e2.3}.
We consider a random variable $F\in L_\xi$ with $\BE F^2<\infty$
and denote by $N$ a standard normal
random variable. In this section we derive upper bounds for the Wasserstein distance
$$
d_1(F,N):=\sup_{h\in {\operatorname{Lip}}(1)} |\BE h(F)-\BE h(N)|
$$
between $F$ and $N$,
where ${\operatorname{Lip}}(1)$ is the set of all functions
$h:\R\to\R$ with a Lipschitz constant less than or equal to one.
We also study the Kolmogorov distance
\begin{align}\label{eKolmogorov}
d_K(F,N):=\sup_{t \in \R} |\BP(F\leq t)-\BP(N\leq t)|.
\end{align}
Our aim is to extend Theorems 1.1 and 1.2 from \cite{LaPeSchu16}, treating functions of Poisson processes,
to the present (more general) setting of underlying edge marked Poisson processes.

Our bounds on the Wasserstein distance are based on the
following three terms:
\begin{align*}
\gamma_1 & :=  2\bigg[\int \big[\BE (\Delta_{x_1}F)^2 (\Delta_{x_2}F)^2\big]^{1/2}
\Big[\BE \big(\Delta_{x_1,x_3}^2F\big)^2\big(\Delta_{x_2,x_3}^2F\big)^2\Big]^{1/2}\, \lambda^3(d(x_1,x_2,x_3))
\bigg]^{1/2}, \allowdisplaybreaks \\
\gamma_2 & := \bigg[\int \BE \big(\Delta_{x_1,x_3}^2F\big)^2\big(\Delta_{x_2,x_3}^2F\big)^2 \,\lambda^3(d(x_1,x_2,x_3))
\bigg]^{1/2}, \allowdisplaybreaks \\
\gamma_3 & := \int \BE |\Delta_xF|^3 \, \lambda(dx).
\end{align*}
The bounds on the Kolmogorov distance involve three more terms:
\begin{align*}
\gamma_4 & := \frac{1}{2} \big[\BE F^4\big]^{1/4} \int \big[\BE (\Delta_xF)^4\big]^{3/4} \, \lambda(dx), \allowdisplaybreaks\\
\gamma_5 & := \bigg[\int \BE(\Delta_xF)^4 \, \lambda(dx)\bigg]^{1/2}, \allowdisplaybreaks\\
\gamma_6 & := \bigg[\int 6\big[\BE(\Delta_{x_1}F)^4\big]^{1/2} \Big[\BE\big(\Delta^2_{x_1,x_2}F\big)^4\Big]^{1/2}
+3\,\BE\big(\Delta^2_{x_1,x_2}F\big)^4 \, \lambda^{2}(d(x_1,x_2))\bigg]^{1/2}.
\end{align*}
Since the quantities $\gamma_1,\hdots,\gamma_6$ depend basically only on the first two difference operators, the following result says that the first two difference operators are sufficient to control the closeness to a standard normal random variable $N$.

\begin{theorem}\label{thm:General.CLT}
Let $F\in L_\xi$ be
such  that $\BE F=0$, $\V F=1$ and $\BE F^4<\infty$.
Assume also that $\gamma_5,\gamma_6<\infty$, that
\begin{equation}\label{eqn:Integrability1}
\int \big[\BE (\Delta_{x_1}F)^4\big]^{1/4} \big[\BE (\Delta_{x_2}F)^4\big]^{1/4}
\Big[\BE \big(\Delta_{x_1,x_3}^2F\big)^4\Big]^{1/4}\Big[\BE  \big(\Delta_{x_2,x_3}^2F\big)^4\Big]^{1/4}\, \lambda^3(d(x_1,x_2,x_3)) < \infty
\end{equation}
and that
\begin{equation}\label{eqn:Integrability2}
\int \Big[\BE \big(\Delta_{x_1,x_3}^2F\big)^4\Big]^{1/2} \Big[\BE\big(\Delta_{x_2,x_3}^2F\big)^4\Big]^{1/2} \,\lambda^3(d(x_1,x_2,x_3) <\infty.
\end{equation}
Then
$$
d_1(F,N)\le\gamma_1+\gamma_2+\gamma_3 \quad \text{and} \quad d_K(F,N)\le\gamma_1+\gamma_2+\gamma_3+\gamma_4+\gamma_5+\gamma_6.
$$
\end{theorem}

We prepare the proof of Theorem \ref{thm:General.CLT} with the following lemma.

\begin{lemma}\label{lem:Approximation}
Let $Y_n$, $n\in\N$, and $Y$ be square integrable random variables such that $Y_n\to Y$ in $L^2(\P)$ as $n\to\infty$. Then
$$
d_1(Y,N) \leq \liminf_{n\to\infty} d_1(Y_n,N) \quad \text{ and } \quad d_K(Y,N) \leq \liminf_{n\to\infty} d_K(Y_n,N).
$$
\end{lemma}

\begin{proof}
The assertion for the Wasserstein distance is easy to see since
$$
d_1(Y,N) \leq d_1(Y,Y_n) + d_1(Y_n,N) \leq \BE |Y-Y_n| + d_1(Y_n,N) \leq \sqrt{\BE |Y-Y_n|^2} + d_1(Y_n,N).
$$

We define $\varepsilon_n:=\BE |Y-Y_n|^2$, $n\in\N$, so that $\varepsilon_n\to0$ as $n\to\infty$. For a fixed $t\in\R$ we bound
$$
|\P(Y\leq t) -\P(N\leq t)|
$$
in the following. If the difference is positive, we have that
\begin{align*}
\P(Y\leq t) -\P(N\leq t) & \leq \P\big(Y_n \leq t+\varepsilon_n^{1/3}\big) + \P(|Y-Y_n|\geq \varepsilon_n^{1/3}) -\P(N\leq t)\\
& = \P\big(Y_n \leq t+\varepsilon_n^{1/3}\big) - \P\big(N\leq t+\varepsilon_n^{1/3}\big)+ \P\big(|Y-Y_n|\geq \varepsilon_n^{1/3}\big)\\
& \quad + \P\big(N\leq t+\varepsilon_n^{1/3}\big)-\P(N\leq t)\\
& \leq d_K(Y_n,N)+ \frac{\varepsilon_n}{\varepsilon_n^{2/3}} + \frac{\varepsilon_n^{1/3}}{\sqrt{2\pi}},
\end{align*}
where we used the Markov inequality in the last step. If $\P(Y\leq t) -\P(N\leq t)\leq 0$, we obtain by similar arguments
\begin{align*}
\P(N\leq t) - \P(Y\leq t) & \leq \P(N\leq t) - \P\big(Y_n\leq t-\varepsilon_n^{1/3}\big) + \P\big(|Y-Y_n|\geq \varepsilon_n^{1/3}\big)\\
& \leq d_K(Y_n,N) + 2 \varepsilon_n^{1/3}.
\end{align*}
Altogether we see that
$$
|\P(Y\leq t) -\P(N\leq t)| \leq d_K(Y_n,N)+2\varepsilon_n^{1/3}
$$
so that taking the limit inferior for $n\to\infty$ and the supremum over all $t\in\R$ completes the proof.
\end{proof}

\begin{proof}[Proof of Theorem \ref{thm:General.CLT}]
We first assume that $\lambda(\mathbf{X})<\infty$ and that $F$ is bounded. Let $f$ be a representative of $F$ satisfying $\|f\|_\infty<\infty$. Fix $\varepsilon>0$.
The identity \eqref{eqn:XiEtaStar} shows that
$F\overset{d}{=}f_\varepsilon^*(\eta^*)$,
where $f^*_\varepsilon:=f\circ T_\varepsilon$.
In the following we apply Theorems 1.1 and 1.2 in \cite{LaPeSchu16} to
the Poisson functional $F^*_\varepsilon:=f^*_\varepsilon(\eta^*)$.
Note that the required integrability condition
$$
\iint \BE \big(D_{(x,y)}F^*_\varepsilon\big)^2\,\BQ(dy)\,\lambda(dx)<\infty
$$
is obviously satisfied because of $\|f\|_\infty<\infty$ and $\lambda(\mathbf{X})<\infty$. Thus, Theorems 1.1 and 1.2 in \cite{LaPeSchu16} lead to
\begin{align}\label{e4.16}
d_1(F^*_\varepsilon,N)\le\sum^3_{i=1}\gamma^*_i(\varepsilon) \quad \text{and} \quad
d_K(F^*_\varepsilon,N)\le \sum^6_{i=1}\gamma^*_i(\varepsilon),
\end{align}
where the terms $\gamma^*_i(\varepsilon)$ are suitably defined
versions of the $\gamma_i$ with $\Delta$ replaced by $D$ and $F$ replaced by $F^*_\varepsilon$. For instance we have that
\begin{align*}
\gamma^*_1(\varepsilon) &
=2\bigg[\iint \Big[\BE \big(D_{(x_1,y_1)}F_\varepsilon^*\big)^2 \big(D_{(x_2,y_2)}F_\varepsilon^*\big)^2\Big]^{1/2}\\
&\quad
\times\Big[\BE \big(D_{(x_1,y_1),(x_3,y_3)}^2F_\varepsilon^*\big)^2\big(D_{(x_2,y_2),(x_3,y_3)}^2F_\varepsilon^*\big)^2\Big]^{1/2}
\,\BQ^3(d(y_1,y_2,y_3))\,\lambda^3(d(x_1,x_2,x_3))\bigg]^{1/2}.
\end{align*}
By the Cauchy-Schwarz inequality we have
$\gamma^*_1(\varepsilon)\le \tilde\gamma_1(\varepsilon)$, where
\begin{align*}
\tilde\gamma_1(\varepsilon) :=2\bigg[\int &\Big[\BE \big(D_{(x_1,Y_1)}F^*_\varepsilon\big)^2
\big(D_{(x_2,Y_2)}F^*_\varepsilon\big)^2\Big]^{1/2}\\
&\times\Big[\BE \big(D_{(x_1,Y_1),(x_3,Y_3)}^2F_\varepsilon^*\big)^2\big(D_{(x_2,Y_2),(x_3,Y_3)}^2F_\varepsilon^*\big)^2\Big]^{1/2}
\,\lambda^3(d(x_1,x_2,x_3))\bigg]^{1/2}.
\end{align*}
Here and in the following, we denote by $Y,Y_1,Y_2,Y_3$ independent random elements of $\mathbf{M}$ with distribution $\BM$, which are independent
of everything else.
Similarly we can treat the other summands in \eqref{e4.16} to obtain that
\begin{align}\label{e4.32}
d_1(F,N)\le\sum^3_{i=1}\tilde\gamma_i(\varepsilon) \quad \text{and} \quad
d_K(F,N)\le \sum^6_{i=1}\tilde\gamma_i(\varepsilon),
\end{align}
where $\tilde{\gamma}_2(\varepsilon),\ldots,\tilde{\gamma}_6(\varepsilon)$ are defined analogously to $\tilde{\gamma}_1(\varepsilon)$.

In order to proceed from \eqref{e4.32} to the bounds asserted by the theorem,
we shall show that
$\tilde\gamma_i(\varepsilon)\to\gamma_i$ as $\varepsilon\to 0$ for each $i\in[6]$.
The relations \eqref{eqn:IdentityDDelta}, \eqref{eqn:x1..xk} and $\|f\|_\infty<\infty$
easily imply for each $i\in[6]$ that the
integrands in $\tilde{\gamma}_i(\varepsilon)$ converge
as $\varepsilon\to 0$ almost everywhere pointwise to the integrands
in $\gamma_i$. Because the integrands are bounded and $\lambda(\mathbf{X})<\infty$, the dominated convergence theorem yields the desired conclusion for bounded $F$ and $\lambda(\mathbf{X})<\infty$.

Next we consider general $F$, but still assume that $\lambda(\mathbf{X})<\infty$. For $n\in\N$ we define
$$
F_n:= {\bf 1}\{F>n\} n + {\bf 1}\{-n \leq F \leq n \} F -{\bf 1}\{F<-n\}n
$$
and let $\gamma_{i,n}$ be $\gamma_i$ with $F$ replaced by $F_n-\BE F_n$ for $i\in[6]$. Whenever we take difference operators, we can omit the constant $-\BE F_n$. Denoting by $f_n$ a representative of $F_n$, we obtain that
\begin{equation}\label{eqn:DeltaFn}
|\Delta_xF_n| \leq |f_n(\xi_x)| + |f_n(\xi)|\leq |f(\xi_x)| + |f(\xi)| = |\Delta_xF+F| + |F| \leq |\Delta_xF| + 2|F|
\end{equation}
for $x\in\mathbf{X}$ and
\begin{align}
\big|\Delta^2_{x_1,x_2}F_n\big| & \leq |f_n(\xi_{x_1,x_2})|+|f_n(\xi_{x_1})|+|f_n(\xi_{x_2})|+|f_n(\xi)| \notag \\
& \leq |f(\xi_{x_1,x_2})|+|f(\xi_{x_1})|+|f(\xi_{x_2})|+|f(\xi)| \notag \allowdisplaybreaks\\
& = \big|\Delta^2_{x_1,x_2}F+\Delta_{x_1}F+\Delta_{x_2}F+F\big|+|\Delta_{x_1}F+F|+|\Delta_{x_2}F+F| + |F| \notag \\
& \leq \big|\Delta^2_{x_1,x_2}F\big| + 2 |\Delta_{x_1}F| + 2 |\Delta_{x_2}F| + 4|F| \label{eqn:Delta2Fn}
\end{align}
for $x_1,x_2\in\mathbf{X}$. Together with $\lambda(\bX)<\infty$, $\BE F^4<\infty$ and $\gamma_5,\gamma_6<\infty$ this gives us integrable upper bounds for the integrands in $\gamma_{1,n},\hdots,\gamma_{6,n}$ and, thus, allows us to apply the dominated convergence theorem. Since also the integrands of $\gamma_{i,n}$ converge almost everywhere to the integrands of $\gamma_i$ for $i\in[6]$ (this follows from the dominated convergence theorem with the upper bounds given by \eqref{eqn:DeltaFn} and \eqref{eqn:Delta2Fn}) and $\BE(F_n-\BE F_n)^4\to\BE F^4$ as $n\to\infty$, we obtain that
$
\lim_{n\to\infty} \gamma_{i,n} = \gamma_i
$
for $i\in[6]$. We define the bounded functionals $\tilde{F}_n:=(F_n-\BE F_n)/\sqrt{\V F_n}$, $n\in\N$. Note that $F_n\to F$ in $L^2(\P)$ as $n\to\infty$ and, thus, $\V F_n\to\V F=1$ and $\tilde{F}_n\to F$ in $L^2(\P)$ as $n\to\infty$. Applying the derived bounds for bounded functionals to $\tilde{F}_n$ and using Lemma \ref{lem:Approximation} shows the assertion for $\lambda(\mathbf{X})<\infty$.

Finally, we allow for $\lambda(\mathbf{X})=\infty$. Let $A_n\in\mathcal{X}$, $n\in\N$, be increasing and such that $\lambda(A_n)<\infty$, $n\in\N$, and $\bigcup_{n=1}^\infty A_n=\mathbf{X}$. For $n\in\N$ let $\mathcal{A}_n$ be the $\sigma$-field on $\mathbf{N}\big(A_n^{[2]}\times \mathbf{M}\big)$, i.e., the smallest $\sigma$-field such that for any  measurable $B\subset A_n^{[2]}\times \mathbf{M}$ the map $\mathbf{N}\big(A_n^{[2]}\times \mathbf{M}\big)\ni\mu\mapsto \mu(B)$ is measurable. Obviously the $\sigma$-field generated by $\mathcal{A}_n$, $n\in\N$, is the $\sigma$-field on $\mathbf{N}(\bX^{[2]}\times \mathbf{M})$.

Recall that $\mathbf{Y}=\mathbf{M}^{\N\times\N}$. For $n\in\N$ there is a measurable map
$$
T_n: \mathbf{N}\big(A_n^{[2]}\times \mathbf{M}\big)\times \mathbf{N}\big((A_n^c)^{[2]}\times \bM\big) \times \bY \to\mathbf{N}(\bX^{[2]}\times \bM)
 $$
such that for $x_1,\hdots,x_m\in A_n$, $m\in\N_0$,
$$
\xi_{x_1,\hdots,x_m}=T_n\big(\xi_{x_1,\hdots,x_m}\cap A_n^{[2]}\times\bM,\xi \cap (A_n^c)^{[2]}\times\bM,\tilde{M}\big)
$$
with some double sequence $\tilde{M}\in\mathbf{Y}$. In order to obtain $\tilde{M}$, we order the underlying points of $\xi_{x_1,\hdots,x_m}$ in $A_n$ and in $A_n^c$. For $i,j\in\N$ let $\tilde{M}_{ij}$ be the mark associated with the $i$-th point in $A_n$ and the $j$-th point in $A_n^c$ if there are at least $i$ points in $A_n$ and let $\tilde{M}_{ij}$ be some random element of $\bM$ with distribution $\mathbb{M}$, which is independent of everything else, otherwise. Note that $T_n$ does not depend on $m$ or $x_1,\hdots,x_m$. By construction of the edge marking and the independence properties of $\eta$ we have that
$\xi_{x_1,\hdots,x_m}\cap A_n^{[2]}\times\bM$, $\xi\cap (A_n^c)^{[2]}\times\bM$ and $\tilde{M}$ are independent and that $\tilde{M}$ is distributed according to $\mathbb{Q}$. This yields now that
$$
\BE[ f(\xi_{x_1,\hdots,x_m}) \mid \mathcal{A}_n] = f_n\big(\xi_{x_1,\hdots,x_m}\cap A_n^{[2]}\times\bM\big)
$$
with $f_n: \mathbf{N}\big(A_n^{[2]}\times\mathbf{M}\big) \to \R$, $n\in\N$, given by
$$
f_n(\mu):= \BE\big[f\big(T_n\big(\mu,\xi'\cap (A_n^c)^{[2]}\times\bM,Y\big)\big) \big], \qquad \mu\in\mathbf{N}\big(A_n^{[2]}\times\mathbf{M}\big),
$$
where $\xi'$ is an independent copy of $\xi$ and $Y$ is distributed according to $\mathbb{Q}$ and independent of $\xi'$. Now we define the Doob martingale
$$
F_n := \BE[F \mid \mathcal{A}_n], \quad n\in\N,
$$
so that $F_n = f_n(\xi\cap A_n^{[2]}\times\bM)$,
\begin{align}\label{eqn:mating.diff.op.1}
f_n\big(\xi_x\cap A_n^{[2]}\times\bM\big) - f_n\big(\xi\cap A_n^{[2]}\times\bM\big) & = \BE[ f(\xi_{x}) \mid \mathcal{A}_n] - \BE[ f(\xi) \mid \mathcal{A}_n]\nonumber\\
& = \BE[ f(\xi_{x})-f(\xi) \mid \mathcal{A}_n]= \BE[ \Delta_xF \mid \mathcal{A}_n]
\end{align}
for $x\in A_n$ and
\begin{align}\label{eqn:mating.diff.op.2}
& f_n\big(\xi_{x_1,x_2}\cap A_n^{[2]}\times\bM\big) - f_n\big(\xi_{x_1}\cap A_n^{[2]}\times\bM\big) - f_n\big(\xi_{x_2}\cap A_n^{[2]}\times\bM\big) + f_n(\xi\cap A_n^{[2]}\times\bM)\nonumber\\
 & = \BE[ f(\xi_{x_1,x_2}) \mid \mathcal{A}_n] - \BE[ f(\xi_{x_1}) \mid \mathcal{A}_n]  - \BE[ f(\xi_{x_2}) \mid \mathcal{A}_n] + \BE[ f(\xi) \mid \mathcal{A}_n]\nonumber\\
 & = \BE[ f(\xi_{x_1,x_2}) - f(\xi_{x_1}) - f(\xi_{x_2}) +  f(\xi) \mid \mathcal{A}_n] = \BE\big[ \Delta^2_{x_1,x_2}F \bigm| \mathcal{A}_n\big]
\end{align}
for $x_1,x_2\in A_n$.

We can think of $F_n$ as a functional of the edge marking $\xi^{(n)}$ with respect to a Poisson process with intensity measure $\lambda|_{A_n}$, i.e., the restriction of $\lambda$ to $A_n$, and mark distribution $\mathbb{M}$ and denote by $\Delta^m_{x_1,\hdots,x_m}F_n$, $x_1,\hdots,x_m\in A_n$, its $m$-th difference operator applied to $F_n$. Note that for $x_1,\hdots,x_m\in A_n$, $m\in\N$,
$$
\big(\xi^{(n)}_{(x_i)_{i\in I}}\big)_{I\subset [m]} \overset{d}{=} \big(\xi_{(x_i)_{i\in I}}\cap A_n^{[2]}\times\bM\big)_{I\subset[m]}.
$$
Hence \eqref{eqn:mating.diff.op.1} and \eqref{eqn:mating.diff.op.2} imply that, for $x_1,x_2,x_3\in A_n$,
\begin{align} \notag \label{eqn:xiA}
& \Big((\Delta_{x_i}F_n)_{i\in[3]}, \big(\Delta^2_{x_i,x_j}F_n\big)_{i,j\in[3],i\neq j}\Big)\\
& \overset{d}{=} \Big((\BE[\Delta_{x_i}F \mid \mathcal{A}_n])_{i\in[3]}, \big(\BE\big[\Delta^2_{x_i,x_j}F \bigm| \mathcal{A}_n\big]\big)_{i,j\in[3],i\neq j}\Big).
\end{align}
The Jensen inequality implies $\BE F_n^4 \leq \BE F^4$. Hence the martingale convergence theorem (see \cite[Corollary 7.22 and Theorem 7.23]{Kallenberg2002}) yields that
$$
F_n \to F \quad \text{ in $L^4(\P)$ } \quad \text{ as } \quad n\to\infty.
$$
For $n\in\N$ denote by $\gamma_{i,n}$, $i\in[6]$, $\gamma_i$ with $F$ replaced by $F_n$ and all integrations with respect to $\lambda|_{A_n}$. This is further rewritten by using the identity \eqref{eqn:xiA}.
 For $p\in [1,\infty)$ and $x,x_1,x_2\in A_n$, the Jensen inequality implies that
$$
\BE |\BE[\Delta_xF \mid \mathcal{A}_n]|^p \leq \BE |\Delta_xF|^p \quad \text{and} \quad
 \BE \big|\BE\big[ \Delta^2_{x_1,x_2}F \bigm| \mathcal{A}_n\big] \big|^p \leq \BE \big|\Delta^2_{x_1,x_2}F\big|^p,
$$
whence we have integrable upper bounds for the integrands of $\gamma_{3,n},\hdots,\gamma_{6,n}$. For  $\gamma_{1,n}$ and $\gamma_{2,n}$ we obtain such bounds by the Cauchy-Schwarz inequality, the previous inequalities and the assumptions \eqref{eqn:Integrability1} and \eqref{eqn:Integrability2}.
By the martingale convergence theorem we have that, for $\lambda$-a.e.\ $x\in\mathbf{X}$,
$$
\BE[\Delta_xF \mid \mathcal{A}_n] \to \Delta_xF \quad \text{ in $L^4(\P)$ } \quad \text{ as } \quad n\to\infty
$$
and that, for $\lambda^2$-a.e.\ $(x_1,x_2)\in\mathbf{X}^2$,
$$
\BE\big[\Delta^2_{x_1,x_2}F \bigm| \mathcal{A}_n\big] \to \Delta^2_{x_1,x_2}F \quad \text{ in $L^4(\P)$ } \quad \text{ as } \quad n\to\infty.
$$
Consequently, the integrands of $\gamma_{1,n},\hdots,\gamma_{6,n}$ converge almost everywhere to those of $\gamma_1,\hdots,\gamma_6$. Now the dominated convergence theorem yields that
$\lim_{n\to\infty} \gamma_{i,n} = \gamma_i$
for $i\in[6]$. Applying the bounds for finite measures to $\tilde{F}_n:=(F_n-\BE F_n)/\sqrt{\V F_n}$, $\V F_n\to\V F$ as $n\to\infty$ and Lemma \ref{lem:Approximation} complete the proof.
\end{proof}

The following lemma allows us to bound the fourth moment in $\gamma_4$ in terms of the first difference operator.

\begin{lemma}\label{lem:4.moment}
Let $F\in L_\xi$ be such  that $\BE F^4<\infty$, $\BE F=0$ and $\V F=1$. Then
\begin{align*}
  \BE F^4 \le \max \bigg\{ 256 \bigg[ \int \big[ \BE \left( \Delta_x
          F \right)^4 \big]^{1/2} \lambda (dx) \bigg]^2,
    \; 4 \int \BE \left( \Delta_x F \right)^4 \lambda (d x) +
    2 \bigg\}.
\end{align*}
\end{lemma}

\begin{proof}
Because of the Poincar\'e inequality Theorem \ref{tpoincare} and the product formula $\Delta_x(F^2)=2 F \Delta_xF+(\Delta_xF)^2$, $x\in\mathbf{X}$, this can be shown exactly as Lemma 4.2 in \cite{LaPeSchu16} if one replaces $D_x$ by $\Delta_x$ everywhere in the proof.
\end{proof}

\section{Asymptotic covariances}\label{sasymptoticcov}
In this section we consider the RCM $\Gamma(\eta)$ based on a stationary
Poisson process $\eta$ on $\R^d$ with intensity $\beta>0$. The connection function $\varphi\colon(\R^d)^2\to[0,1]$
is not only assumed to be measurable and symmetric but also to be translation invariant, that is we have
$\varphi(x,y)=\varphi(0,y-x)$ for all $x,y\in\R^d$.
By an abuse of notation we set $\varphi:=\varphi(0,\cdot)$ and
note that  $\varphi(x)=\varphi(-x)$, $x\in\R^d$. Throughout this section we shall assume that
\begin{align}\label{e5.1}
0<m_\varphi := \int \varphi(x) \, dx<\infty.
\end{align}
For $m_\varphi=0$ or $m_\varphi=\infty$ the component counts become trivial since for $m_\varphi=0$ each vertex of the RCM is isolated almost surely, while for $m_\varphi=\infty$ each vertex has infinitely many neighbours almost surely.
As transitive binary relation $\prec$ we use the lexicographic order on $\R^d$, which is translation invariant in the sense that $x+z\prec y+z$ for all $x,y,z\in\R^d$ with $x\prec y$. Recall that $r(W)$ stands for the inradius of $W\in\mathcal{K}^d$.
In what follows we
consider sequences of convex bodies $(W_n)_{n\in\N}$ such that
$r(W_n)\to\infty$ as $n\to\infty$.
We denote this asymptotic regime by $r(W)\to\infty$.
We are interested in the asymptotic covariances
\begin{align}\label{asGH}
\sigma_{\varphi,\psi}(G,H):=\lim_{r(W)\to\infty} \frac{\C(\eta_{\varphi,G}(W),\eta_{\psi,H}(W))}{\lambda_d(W)},
\end{align}
where $G,H\in\bG$ are finite connected graphs and $\psi$ is a second connection function with the same properties as $\varphi$. For the special case of random geometric graphs, where $\varphi(x)=\psi(x)=\mathbf{1}\{|x|\leq r\}$, $x\in\R^d$, for some $r>0$, such asymptotic covariances are computed in \cite[Proposition 3.8]{Penrose03}. Recall the definition of $\Gamma_\varphi(\zeta)$ for any simple point process $\zeta$ on $\R^d$ and the definition of $\mathbf{G}_k$, $k\in\N$, in Section \ref{secpre} as well as \eqref{e3.2}.

\begin{theorem}\label{thm:AsymptoticVariance}
Let $k,l\in\N$, $G\in\bG_k$ and $H\in\bG_l$.  Assume \eqref{e5.1} and that $\psi\le\varphi$.
Then the limit \eqref{asGH} exists as a finite number and is given by
\begin{align*}
\sigma_{\varphi,\psi}(G,H)&=
\beta^{k+l}\int p_{\varphi,G}(0,x_2,\ldots,x_k)p_{\psi,H}(x_{k+1},\ldots,x_{k+l})
q_{k,l,\varphi,\psi}(x_2,\ldots,x_{k+l})\,d(x_2,\dots,x_{k+l})\\
&\quad +\I\{k=l\}\beta^k\int \BP(\Gamma_\varphi(\{x_1,\ldots,x_k\})\simeq G, \Gamma_\psi(\{x_1,\ldots,x_k\})\simeq H)\\
& \hskip 1.5cm
\times \1\{0\prec x_2\prec\cdots\prec x_k\} \exp\Bigg[\beta \int \Bigg(\prod^k_{i=1}\bar\varphi(x_i-y)-1\Bigg)\,dy\Bigg]\,d(x_2,\dots,x_{k}),
\end{align*}
where $p_{\varphi,G}$ and $p_{\varphi,H}$ are as in \eqref{e3.2},
\begin{align*}
q_{k,l,\varphi,\psi}(x_2,\ldots,x_{k+l})&:=\prod^k_{i=1}\prod^{k+l}_{j=k+1}\bar\varphi(x_i-x_j)
\exp\Bigg[\beta\int \Bigg(\prod^k_{i=1}\bar\varphi(x_i-y)\prod^{k+l}_{j=k+1}\bar\psi(x_j-y)-1\Bigg)\,dy\Bigg]\\
&\quad-\exp\Bigg[\beta\int \Bigg(\prod^k_{i=1}\bar\varphi(x_i-y)-1\Bigg)\,dy + \beta\int \Bigg(\prod^{k+l}_{j=k+1}\bar\psi(x_j-y)-1\Bigg)\,dy\Bigg]
\end{align*}
and $x_1:=0$.
\end{theorem}

\begin{proof}
Let $W\in\mathcal{K}^d$. By translation invariance of $\varphi$, we have
\begin{align*}
p_{\varphi,G}(x_1,\ldots,x_k)=p_{\varphi,G}(0,x_2-x_1,\ldots,x_k-x_1),\quad x_1,\ldots,x_k\in\R^d.
\end{align*}
Therefore it follows from
Proposition \ref{p3.7} and translation invariance of Lebesgue measure that
\begin{align*}
&\BE\eta_{\varphi,G}(W)\eta_{\psi,H}(W)=\beta^{k+l}\iint\I\{z\in W,z+x_{k+1}\in W\}p_{\varphi,G}(0,x_2,\ldots,x_k)
p_{\psi,H}(x_{k+1},\ldots,x_{k+l})\\
&\times \prod^k_{i=1}\prod^{k+l}_{j=k+1}\bar\varphi(x_i-x_j)
\exp\Bigg[\beta\int \Bigg(\prod^k_{i=1}\bar\varphi(x_i-y)\prod^{k+l}_{j=k+1}\bar\psi(x_j-y)-1\Bigg)\,
dy\Bigg]\,dz\,d(x_2,\dots,x_{k+l})\\
&+\I\{k=l\}\beta^k\iint\I\{z\in W\} \BP(\Gamma_\varphi(\{0,x_2,\ldots,x_k\})\simeq G, \Gamma_\psi(\{0,x_2,\ldots,x_k\})\simeq H) \\
& \hskip 2cm
\times \1\{0\prec x_2\prec\cdots\prec x_k\} \exp\Bigg[\beta\int \Bigg(\prod^k_{i=1}\bar\varphi(x_i-y)-1\Bigg)\,dy\Bigg]\,dz\,d(x_2,\dots,x_k),
\end{align*}
where $x_1:=0$. By a similar calculation for the product
$\BE\eta_{\varphi,G}(W)\BE\eta_{\psi,H}(W)$ (see \eqref{intcluster}) we obtain that
\begin{align*}
&\C(\eta_{\varphi,G}(W),\eta_{\psi,H}(W)) \notag \\
&=\beta^{k+l}\int\lambda_d(W\cap (W-x_{k+1}))p_{\varphi,G}(0,x_2,\ldots,x_k)
p_{\psi,H}(x_{k+1},\ldots,x_{k+l}) \notag \\
&\hskip 7cm\times q_{k,l,\varphi,\psi}(x_2,\ldots,x_{k+l})\,d(x_2,\dots,x_{k+l}) \notag \\
&\qquad +\I\{k=l\}\beta^k\lambda_d(W)\int \BP(\Gamma_\varphi(\{0,x_2,\ldots,x_k\})\simeq G, \Gamma_\psi(\{0,x_2,\ldots,x_k\})\simeq H) \notag \\
& \hskip 2cm \times \1\{0\prec x_2\prec\cdots\prec x_k\} \exp\Bigg[\beta\int \Bigg(\prod^k_{i=1}\bar\varphi(x_i-y)-1\Bigg)\,dy\Bigg]\,d(x_2,\dots,x_k).
\end{align*}

First we show that the integral in the second summand is finite. For $n\in\N$ we use the abbreviation $\mathbf{I}_n := \{\{i,j\}\subset[n] : i<j\}$
and define
$$
\mathcal{I}_n := \{I \subset \mathbf{I}_n : |I|=n-1 \text{ and the graph $([n], I)$ is connected} \}.
$$
For $y_1,\ldots,y_n\in\R^d$ we have that
$$
\BP(\Gamma_\varphi(\{y_1,\ldots,y_n\})\text{ is connected}) \leq \sum_{I\in\mathcal{I}_n} \prod_{(i,j)\in I}\varphi(y_i-y_j).
$$
Since the integrand is bounded by the probability that $\Gamma_\varphi(\{0,x_2,\hdots,x_k\})$ is connected, the integral in the second summand of the covariance representation is bounded by
$$
\int \sum_{I\in\mathcal{I}_k} \prod_{(i,j)\in I}\varphi(x_i-x_j) \,d(x_2,\dots,x_k) = |\mathcal{I}_n| m_\varphi^{k-1}.
$$
In the last equation we used that the graph $([k], I)$ is a tree for any $I\in\mathcal{I}_k$ and integrated successively beginning with the variables whose indices are leaves of the tree.

For the rest of this proof we consider the integral in the first summand of the above covariance formula.
It is a standard fact from stochastic geometry (see e.g.\ \cite[p.\ 88]{HugLastSchulte2016}) that
$$
\lim_{r(W)\to\infty} \frac{\lambda_d(W\cap (W-x))}{\lambda_d(W)}=1
$$
for any fixed $x\in\R^d$.

Next we bound
$$
h(x_2,\ldots,x_{k+l}) := p_{\varphi,G}(0,x_2,\ldots,x_k)
p_{\psi,H}(x_{k+1},\ldots,x_{k+l}) \, |q_{k,l,\varphi,\psi}(x_2,\ldots,x_{k+l})|
$$
with $x_2,\ldots,x_{k+l} \in\R^d$
by an integrable function so that the assertion follows from the dominated convergence theorem.
For simplicity we assume now that $\beta=1$.

Using the same notation as above we have
$$
p_{\varphi,G}(y_1,\ldots,y_k) \leq \BP(\Gamma_\varphi(\{y_1,\ldots,y_k\})\text{ is connected}) \leq \sum_{I\in\mathcal{I}_k} \prod_{(i,j)\in I}\varphi(y_i-y_j), \quad y_1,\ldots,y_k\in\R^d,
$$
and analogously
\begin{align*}
p_{\psi,H}(y_1,\ldots,y_l) &\leq \BP(\Gamma_\psi(\{y_1,\ldots,y_l\})\text{ is connected})\\
&\leq \sum_{I\in\mathcal{I}_l} \prod_{(i,j)\in I} \psi(y_i-y_j)
\leq \sum_{I\in\mathcal{I}_l} \prod_{(i,j)\in I} \varphi(y_i-y_j), \qquad y_1,\ldots,y_l\in\R^d.
\end{align*}
For all $x_2,\ldots,x_{k+l}\in\R^d$ we obtain that
\begin{align*}
|q_{k,l,\varphi,\psi}(x_2,\ldots,x_{k+l})|
\le  &\Bigg|\prod_{i=1}^k\prod_{j=k+1}^{k+l}\bar\varphi(x_i-x_j)-1\Bigg| \\
  & +\Bigg|\exp\Bigg[\int\Bigg(\prod_{i=1}^{k} \bar\varphi(y-x_i)\prod_{j=k+1}^{k+l}
\bar\psi(y-x_j)-1\Bigg) \, dy \Bigg]\\
  & \quad \quad - \exp\Bigg[\int \Bigg(\prod_{i=1}^k\bar\varphi(y-x_i)
+\prod_{j=k+1}^{k+l} \bar\psi(y-x_j)-2\Bigg) \, dy \Bigg] \Bigg|
\end{align*}
with $x_1:=0$. For all $n\in\N$ and all $a_1,\ldots,a_n\in[0,1]$ we have the inequality
\begin{align}\label{e5.5}
1-\prod^{n}_{i=1}(1-a_i)\le \sum^{n}_{i=1}a_i.
\end{align}
Moreover, by the mean value theorem, it holds that
\begin{align*}
|e^{-a}-e^{-b}|\le |a-b|,\quad a,b\ge 0.
\end{align*}
Combining these inequalities yields
\begin{align*}
|&q_{k,l,\varphi,\psi}(x_2,\ldots,x_{k+l})|\le \sum_{i=1}^k\sum_{j=k+1}^{k+l}\varphi(x_i-x_j)\\
&\quad+\Bigg|\int\Bigg(\prod_{i=1}^{k} \bar\varphi(y-x_i)\prod_{j=k+1}^{k+l}\bar\psi(y-x_j)
-\prod_{i=1}^k\bar\varphi(y-x_i)-\prod_{j=k+1}^{k+l} \bar\psi(y-x_j)+1\Bigg)\, dy\Bigg|\\
&=\sum_{i=1}^k\sum_{j=k+1}^{k+l}\varphi(x_i-x_j)
+\Bigg|\int\Bigg(1-\prod_{i=1}^{k} \bar\varphi(y-x_i)\Bigg)
\Bigg(1-\prod_{j=k+1}^{k+l} \bar\psi(y-x_j)\Bigg)\,dy\Bigg|.
\end{align*}
Using \eqref{e5.5} again and then the inequality $\psi\le\varphi$ gives
\begin{align*}
|q_{k,l,\varphi,\psi}(x_2,\ldots,x_{k+l})|
&\le \sum_{i=1}^k\sum_{j=k+1}^{k+l}\varphi(x_i-x_j)
+\int\Bigg(\sum_{i=1}^{k}\varphi(y-x_i)\Bigg)
\Bigg(\sum_{j=k+1}^{k+l} \varphi(y-x_j)\Bigg)\,dy.
\end{align*}
Thus, to verify the integrability of $h$ it suffices to show
for all $i\in[k]$, $j\in\{k+1,\ldots,k+l\}$, $I\in\mathcal{I}_k$ and $J\in\mathcal{I}_l$
that
\begin{align}\label{e5.8}
\int\varphi(x_i-x_j) \prod_{(m,n)\in I\cup \{(j_1+k,j_2+k): (j_1,j_2)\in J\}}\varphi(x_m-x_n) \,d(x_2,\ldots,x_{k+l}) < \infty
\end{align}
and
\begin{align}\label{e5.10}
\iint \varphi(y-x_i) \varphi(y-x_j) \prod_{(m,n)\in I\cup \{(j_1+k,j_2+k): (j_1,j_2)\in J\}}\varphi(x_m-x_n) \,d(x_2,\ldots,x_{k+l}) \,d y < \infty.
\end{align}
Performing the integrations in the right order, we obtain that the left-hand side of \eqref{e5.8} equals $m_\varphi^{k+l-1}$ and the left-hand side of \eqref{e5.10} equals $m_\varphi^{k+l}$.
\end{proof}

Next we consider the covariance structure of $k$-component counts. The asymptotic covariances
\begin{align}\label{asv1}
\sigma^{(k,l)}_{\varphi,\psi}
:=\lim_{r(W)\to\infty} \frac{\C (\eta_{\varphi,k}(W),\eta_{\psi,l}(W))}{\lambda_d(W)},\quad k,l\in\N,
\end{align}
can be expressed in terms of the functions
$$
p_{\varphi,k}(x_1,\ldots,x_k)
:=\I\{x_1\prec\cdots \prec x_k\}\BP(\text{$\Gamma_\varphi(\{x_1,\hdots,x_k\})$ is connected}),
\quad x_1,\ldots,x_k\in\R^d,
$$
as stated in the following corollary, which is a consequence of Theorem \ref{thm:AsymptoticVariance}.

\begin{corollary}\label{thm:AsymptoticVariance2}
If \eqref{e5.1} is satisfied and $\psi\leq\varphi$, for each $k,l\in\N$ the limit \eqref{asv1} exists and is given by
\begin{align*}
\sigma^{(k,l)}_{\varphi,\psi}
&=\beta^{k+l}\int p_{\varphi,k}(0,x_2,\ldots,x_k)p_{\psi,k}(x_{k+1},\ldots,x_{k+l})
q_{k,l,\varphi,\psi}(x_2,\ldots,x_{k+l})\,d(x_2,\dots,x_{k+l})\\
&\quad+\I\{k=l\}\beta^k\int p_{\varphi,k}(0,x_2,\ldots,x_k)
\exp\Bigg[\beta\int \Bigg(\prod^k_{i=1}\bar\varphi(x_i-y)-1\Bigg)\,dy\Bigg]\,d(x_2,\dots,x_k),
\end{align*}
where the function $q_{k,l,\varphi,\psi}$ is as in Theorem \ref{thm:AsymptoticVariance} and $x_1:=0$.
\end{corollary}

For the case $k=l=1$ and $\varphi=\psi$ the formula from Corollary \ref{thm:AsymptoticVariance2} is shown in \cite[Lemma 3.3]{{BrugMeester04}} under slightly stronger assumptions on $\varphi$.
Note that Theorem \ref{thm:AsymptoticVariance} and Corollary \ref{thm:AsymptoticVariance2} immediately yield weak laws of large numbers.

Recall that $\mathbf{G}^{m,\neq}_\varphi$ is the set of all $m$-tuples of distinct graphs from $\mathbf{G}$ that occur as components in $\Gamma_\varphi(\eta)$ with positive probability (see \eqref{eqn:set.G.ne}). For $m\in\N$, $G=(G_1,\hdots,G_m)\in\mathbf{G}^{m,\neq}_\varphi$ and $a=(a_1,\hdots,a_m)\in\R^m$ we define
\begin{equation}\label{eqn:functional.S.a.G}
S_{a,G}(W):=S_{\varphi,a,G}(W):=\sum_{i=1}^m a_i \eta_{\varphi,G_i}(W).
\end{equation}
Furthermore, we let $|a|_\infty := \max \{|a_i| : i \in [m]\}$ for $a=(a_1,\ldots,a_m)\in\R^m$.

\begin{theorem}\label{thm:VarianceS}
Assume that \eqref{e5.1} is satisfied.  For any $m\in\N$, $G=(G_1,\hdots,G_m)\in\mathbf{G}_\varphi^{m,\neq}$ and
  $a=(a_1,\hdots,a_m)\in\R^m$ with $a\neq 0$,
$$
\lim_{r(W)\to\infty} \frac{\V S_{a,G}(W)}{\lambda_d(W)} = \sum_{i,j=1}^m a_i a_j\sigma_{\varphi,\varphi}(G_i,G_j)>0
$$
with $\sigma_{\varphi,\varphi}(G_i,G_j)$ given in Theorem \ref{thm:AsymptoticVariance}.
\end{theorem}

As an immediate consequence of Theorem \ref{thm:VarianceS} we obtain the positive definiteness of some asymptotic covariance matrices.

\begin{corollary}
Let \eqref{e5.1} be satisfied.
\begin{itemize}
\item [(a)] For all $m\in\N$ and $G=(G_1,\hdots,G_m)\in\mathbf{G}_\varphi^{m,\ne}$ the matrix
  $\big(\sigma_{\varphi,\varphi}(G_i,G_j)\big)_{i,j\in[m]}$ given
  in Theorem \ref{thm:AsymptoticVariance} is positive definite.
\item [(b)] For all $m\in\N$ and distinct $k_1,\hdots,k_m\in\N$, the matrix $\big(\sigma_{\varphi,\varphi}^{(k_i,k_j)}\big)_{i,j\in[m]}$ given in Corollary \ref{thm:AsymptoticVariance2} is positive definite.
\end{itemize}
\end{corollary}

The following corollary of Theorem \ref{thm:VarianceS} provides a lower variance bound.

\begin{corollary}\label{cor:lowerBoundVar}
Let the assumptions of Theorem \ref{thm:VarianceS} prevail. Then there exists a constant $\tau>0$ only depending on $\beta$, $\varphi$, $a$ and $G$ such that
$$
\frac{\V S_{a,G}(W)}{\lambda_d(W)} \geq \frac{1}{2} \sum_{i,j=1}^m a_i a_j\sigma_{\varphi,\varphi}(G_i,G_j)>0
$$
for all $W\in \mathcal{K}^d$ with $r(W)\geq \tau$.
\end{corollary}

\begin{proof}[Proof of Theorem \ref{thm:VarianceS}]
Since Theorem \ref{thm:AsymptoticVariance} implies the existence of the limit, it is sufficient to show that
$$
\liminf_{r\to\infty} \frac{\V S_{a,G}(B^d(0,r))}{\lambda_d(B^d(0,r))}>0.
$$
For $r>0$ let $f_r$ be a representative of $S_{a,G}(B^d(0,r))$. Let $\hat{\eta}$ be a Poisson process on $\R^d\times[0,1]\times[0,1]^{\mathbb{N}\times\mathbb{N}}$ with intensity measure $\beta \lambda_d\otimes \lambda_1|_{[0,1]}\otimes \Lambda$ with $\Lambda:=\big(\lambda_1|_{[0,1]}\big)^{\mathbb{N}\times\mathbb{N}}$ and let $\hat{\eta}_t$ denote its restriction to $\R^d\times[0,t)\times[0,1]^{\mathbb{N}\times\mathbb{N}}$ for $t\in[0,1]$. In the following we shall use that $S_{a,G}(B^d(0,r))$ has the same distribution as $f_r(T(\hat{\eta}))$ with $T$ as in Section \ref{secdiffop} and apply Theorem \ref{thm:VarianceRepresentation}. For $t\in[0,1]$ let $\widetilde{\Gamma}(\hat{\eta}_t)$ be the RCM derived from the points of $\hat{\eta}_t$, i.e., from the edge marking $T(\hat{\eta}_t)$. For $(x,t,M)\in \R^d\times[0,1]\times[0,1]^{\mathbb{N}\times\mathbb{N}}$ we denote by $\widetilde{\Gamma}(\hat{\eta}_t\cup\{(x,t,M)\})$ the RCM with respect to the points of $\hat{\eta}_t\cup\{(x,t,M)\}$.

Choose $G_{\max}$ from $G$ such that no other graph of $G$ has more vertices. Let $a_{\max}$ be the weight corresponding to $G_{\max}$ and assume that $G_{\max}$ has $k$ vertices.
Without loss of generality we can assume that $a_{\max}\ne0$. Moreover, we suppose that $a_{\max}<0$ (otherwise we could flip the sign of $a$).

For $(x,t,M)\in\R^d\times[0,1]\times[0,1]^{\mathbb{N}\times\mathbb{N}}$ we denote by $A(x,t,M)$ the event that there are distinct vertices $x_1,\hdots,x_k$ in $\widetilde{\Gamma}(\hat{\eta}_t\cup\{(x,t,M)\})$ with $x_1\in B^d(0,r)$ and $x_1\prec\cdots\prec x_k$ forming a component isomorphic to $G_{\max}$ in $\widetilde{\Gamma}(\hat{\eta}_t\cup\{(x,t,M)\})$ without $x$ and that $x$ is connected to at least one of the vertices $x_1,\hdots,x_k$ and to no other vertex in $\widetilde{\Gamma}(\hat{\eta}_t\cup\{(x,t,M)\})$. Furthermore, let $\tilde{A}(x,t,M)\subset A(x,t,M)$ be the event that $\{x,x_1,\hdots,x_k\}$ is a component of $\widetilde{\Gamma}(\hat{\eta}\cup\{(x,t,M)\})$.
Note that, in contrast to $\tilde{A}(x,t,M)$, the event $A(x,t,M)$ is measurable with respect to $\sigma(\hat\eta_t)$.

It follows from Theorem \ref{thm:VarianceRepresentation} that
\begin{align*}
& \V S_{a,G}(B^d(0,r)) = \V f_r(T(\hat{\eta})) \\
 & \geq \beta \int_{B^d(0,2r)} \int_0^1  \int\BE\big[ \BE\big[ f_r(T(\hat{\eta}\cup\{(x,t,M)\})) - f_r(T(\hat{\eta}\cup\{(x,t,M)\})\setminus\{x\}) | \hat{\eta}_t \big]^2\mathbf{1}_{A(x,t,M)} \big]\\
 & \hspace{3.5cm} \times \Lambda(dM) \, dt \, dx.
\end{align*}
 Next we consider the decomposition
\begin{align*}
& \BE\big[ f_r(T(\hat{\eta}\cup\{(x,t,M)\})) - f_r(T(\hat{\eta}\cup\{(x,t,M)\})\setminus\{x\}) | \hat{\eta}_t \big] \mathbf{1}_{A(x,t,M)}\\
& = \BE\big[ \big(f_r(T(\hat{\eta}\cup\{(x,t,M)\})) - f_r(T(\hat{\eta}\cup\{(x,t,M)\})\setminus\{x\}) \big) \mathbf{1}_{\tilde{A}(x,t,M)} | \hat{\eta}_t \big]  \\
& \quad + \BE\big[ \big(f_r(T(\hat{\eta}\cup\{(x,t,M)\})) - f_r(T(\hat{\eta}\cup\{(x,t,M)\})\setminus\{x\})\big) \mathbf{1}_{\tilde{A}(x,t,M)^c\cap A(x,t,M)} | \hat{\eta}_t \big].
\end{align*}
In the first case removing the vertex $x$ leads to an additional component isomorphic to $G_{\max}$, whence
$$
\BE\big[ \big(f_r(T(\hat{\eta}\cup\{(x,t,M)\})) - f_r(T(\hat{\eta}\cup\{(x,t,M)\})\setminus\{x\}) \big) \mathbf{1}_{\tilde{A}(x,t,M)} | \hat{\eta}_t \big]
 = |a_{\max}| \BP\big( \tilde{A}(x,t,M) | \hat{\eta}_t \big).
$$
In the second case, after deleting $x$, the number of new components that are isomorphic to a graph from $G$ and do not contain $x_1,\ldots,x_k$ is at most the degree of
$x$ in $\widetilde{\Gamma}(\hat{\eta}\cup\{(x,t,M)\})$, minus the degree of $x$ in $\widetilde{\Gamma}(\hat{\eta}_t\cup\{(x,t,M)\})$.
The contribution of each of these components to the difference $f_r(T(\hat{\eta}\cup\{(x,t,M)\})) - f_r(T(\hat{\eta}\cup\{(x,t,M)\})\setminus\{x\})$ is not less than $-|a|_\infty$. Because of $a_{\max}<0$ the contribution of the component containing $x_1,\hdots,x_k$ is non-negative. Together, we see that
\begin{align*}
& \BE\big[ \big( f_r(T(\hat{\eta}\cup\{(x,t,M)\})) - f_r(T(\hat{\eta}\cup\{(x,t,M)\})\setminus\{x\}) \big) \mathbf{1}_{\tilde{A}(x,t,M)^c\cap A(x,t,M)} | \hat{\eta}_t \big] \\
& \geq -|a|_\infty \BE\big[ \big(\operatorname{deg}(x,\widetilde\Gamma(\hat{\eta}\cup\{(x,t,M)\})) - \operatorname{deg}(x,\widetilde\Gamma(\hat{\eta}_t\cup\{(x,t,M)\})) \big) \mathbf{1}_{A(x,t,M)} | \hat{\eta}_t\big].
\end{align*}
With the convention $x_0:=x$, some direct calculations establish that
\begin{align*}
\BP&\big( \tilde{A}(x,t,M) | \hat{\eta}_t \big)\\
& \geq \Bigg(1 - \sum_{j=0}^k \BP\big(x_j \text{ is not connected with } \hat\eta_{[t,1)} \text{ in } \widetilde\Gamma(\hat\eta\cup\{(x,t,M)\}) | \hat{\eta}_t \big)\Bigg)\1_{A(x,t,M)}\\
& = \big( 1 - (k+1) \big(1-\exp(-\beta (1-t) m_\varphi)\big) \big) \mathbf{1}_{A(x,t,M)}\\
& = \big((k+1) \exp(-\beta (1-t) m_\varphi) - k \big) \mathbf{1}_{A(x,t,M)}
\end{align*}
and that
\begin{align*}
 \BE&\big[ \big(\operatorname{deg}(x,\widetilde\Gamma(\hat{\eta}\cup\{(x,t,M)\})) - \operatorname{deg}(x,\widetilde\Gamma(\hat{\eta}_t\cup\{(x,t,M)\})) \big) \mathbf{1}_{A(x,t,M)}  | \hat{\eta}_t\big] \\
& = \beta (1-t) m_{\varphi} \mathbf{1}_{A(x,t,M)}.
\end{align*}
This implies that
\begin{align*}
 \BE&\big[ f_r(T(\hat{\eta}\cup\{(x,t,M)\})) - f_r(T(\hat{\eta}\cup\{(x,t,M)\})\setminus\{x\}) | \hat{\eta}_t \big]  \mathbf{1}_{A(x,t,M)}\\
& \geq \big( |a_{\max}| \big((k+1) \exp(-\beta (1-t) m_\varphi) - k \big) - |a|_\infty \beta (1-t) m_{\varphi} \big) \mathbf{1}_{A(x,t,M)}.
\end{align*}
Now we can choose a $t_0\in[0,1)$ such that
$$
\BE\big[ f_r(T(\hat{\eta}\cup\{(x,t,M)\})) - f_r(T(\hat{\eta}\cup\{(x,t,M)\})\setminus\{x\}) | \hat{\eta}_t \big] \mathbf{1}_{A(x,t,M)} \geq \frac{|a_{\max}|}{2} \mathbf{1}_{A(x,t,M)}
$$
for $t\in[t_0,1]$. Consequently, we have that
\begin{equation}\label{eqn:LowerBoundSaG}
\V f_r(T(\hat{\eta}))  \geq \frac{\beta  |a_{\max}|^2}{4} \int_{B^d(0,2r)} \int_{t_0}^1 \int \BP(A(x,t,M)) \, \Lambda(dM)  \, dt \, dx.
\end{equation}
For $t\in[t_0,1]$ we have that
\begin{align*}
I_{r,t} :& =\int_{B^d(0,2r)} \int \BP(A(x,t,M)) \, \Lambda(dM) \, dx \\
& =  \int_{B^d(0,2r)} \BE \;\sideset{}{^{\ne}}\sum_{((x_1,t_1,M_1),\hdots,(x_k,t_k,M_k))\in\hat{\eta}^k_t} \1\big\{x_1\in B^d(0,r), x_1\prec\hdots\prec x_k\big\} \\
& \hskip 4cm \times\mathbf{1}\{ \{x_1,\hdots,x_k\}\text{ is a component of $\widetilde{\Gamma}(\hat{\eta}_t)$ isomorphic to $G_{\max}$}\} \\
& \hskip 4cm \times\Bigg(1 - \prod_{i=1}^k \bar{\varphi}(x_i-x) \Bigg) \prod_{(z,\hat{t},\hat{M})\in \hat{\eta}_t, z\notin\{x_1,\hdots,x_k\}} \bar{\varphi}(x-z) \, dx.
\end{align*}
As in the proof of Proposition \ref{p1.2}, one obtains that
\begin{align*}
I_{r,t} & =  (\beta t)^k \int \1\big\{x_1\in B^d(0,r),x\in B^d(0,2r)\big\} \, p_{\varphi,G_{\max}}(x_1,\hdots,x_k) \Bigg(1 - \prod_{i=1}^k \bar{\varphi}(x_i-x) \Bigg)\\
& \hskip 3cm \times\exp\Bigg[\beta t  \int\Bigg( \bar{\varphi}(x-z) \prod_{i=1}^k \bar{\varphi}(x_i-z) -1\Bigg) \, dz \Bigg]    d(x_1,\hdots,x_k,x).
\end{align*}
For a fixed $R>0$ and $r\geq R$ we have that
\begin{align*}
I_{r,t} & \geq (\beta t_0)^k \int \1\big\{x_1\in B^d(0,r),|x-x_1|\leq R\big\} \, p_{\varphi,G_{\max}}(x_1,\hdots,x_k) \Bigg(1 - \prod_{i=1}^k \bar{\varphi}(x_i-x) \Bigg)\\
& \hskip 3cm \times\exp\Bigg[ \beta \int\Bigg(\bar{\varphi}(x-z) \prod_{i=1}^k \bar{\varphi}(x_i-z) -1\Bigg) \, dz \Bigg] d(x_1,\hdots,x_k,x),
\end{align*}
whence
\begin{equation}\label{eqn:LowerBoundIr}
\begin{split}
\liminf_{r\to\infty} \frac{I_{r,t}}{\lambda_d(B^d(0,r))} & \geq (\beta t_0)^k \int \mathbf{1}\{|x|\leq R\} \, p_{\varphi,G_{\max}}(0,x_2,\hdots,x_k) \Bigg(1 - \prod_{i=1}^k \bar{\varphi}(x_i-x) \Bigg)\\
& \hskip 2cm \times\exp\Bigg[\beta \int\Bigg(\bar{\varphi}(x-z) \prod_{i=1}^k \bar{\varphi}(x_i-z) -1\Bigg) \, dz \Bigg]    d(x_2,\hdots,x_k,x)
\end{split}
\end{equation}
with the convention $x_1 := 0$.
Since the graph $G_{\max}$ occurs with positive probability in $\Gamma(\eta)$, we obtain
\begin{equation*}
\int  p_{\varphi,G_{\max}}(0,x_2,\hdots,x_k)  \exp\Bigg[ \beta \int\Bigg(\prod_{i=1}^k \bar{\varphi}(x_i-z) -1\Bigg) \, dz \Bigg] d(x_2,\hdots,x_k) >0
\end{equation*}
and hence
$$
\int p_{\varphi,G_{\max}}(0,x_2,\hdots,x_k) \, d(x_2,\hdots,x_k) >0.
$$
Finally, since $0<m_\varphi<\infty$, and letting $x_1:=0$,
$$
1-\bar{\varphi}(x-z) \prod_{i=1}^k \bar{\varphi}(x_i-z) \leq \varphi(x-z) + \sum_{i=1}^k \varphi(x_i-z), \qquad x,z,x_2,\ldots,x_k\in\R^d,
$$
and
$$
1 - \prod_{i=1}^k \bar{\varphi}(x_i-x) \geq \varphi(x_1-x) = \varphi(x), \qquad x,x_2,\ldots,x_k\in\R^d,
$$
we obtain that
\begin{align*}
& \int p_{\varphi,G_{\max}}(0,x_2,\hdots,x_k) \Bigg(1 - \prod_{i=1}^k \bar{\varphi}(x_i-x) \Bigg)\\
& \qquad \times   \exp\Bigg[\beta \int \Bigg(\bar{\varphi}(x-z) \prod_{i=1}^k \bar{\varphi}(x_i-z) -1\Bigg) \, dz \Bigg]    d(x_2,\hdots,x_k,x)>0.
\end{align*}
This implies that the right-hand of \eqref{eqn:LowerBoundIr} is positive for $R$ sufficiently large and completes the proof together with \eqref{eqn:LowerBoundSaG}.
\end{proof}

\section{Normal approximation of component counts}\label{seccounts}

In this section we establish central limit theorems for the
component counts of the RCM $\Gamma(\eta)$ for the setting of Section \ref{sasymptoticcov}. For a positive semidefinite $\Theta\in\R^{m\times m}$, $m\in\N$, we denote by $N_{\Theta}$ a centred $m$-dimensional normal random vector with covariance matrix $\Theta$.

\begin{theorem}\label{thm:CLTSubgraphs}
Let \eqref{e5.1} be satisfied, let $G_1,\hdots,G_m\in\bG$, $m\in\N$, and let $\Sigma=(\sigma_{\varphi,\varphi}(G_i,G_j))_{i,j\in[m]}$ be as in Theorem \ref{thm:AsymptoticVariance}. Then
$$
\frac{1}{\sqrt{\lambda_d(W)}} \big(\eta_{G_1}(W)-\BE
\eta_{G_1}(W),\hdots,\eta_{G_m}(W)-\BE
\eta_{G_m}(W)\big) \overset{d}{\longrightarrow} N_\Sigma \quad \text{as} \quad r(W)\to\infty.
$$
\end{theorem}

For random geometric graphs a multivariate central limit theorem similar to Theorem \ref{thm:CLTSubgraphs} is given in \cite[Theorem 3.11]{Penrose03}. Theorem \ref{thm:CLTSubgraphs} leads to the following central limit theorem for the numbers of $k$-components.

\begin{corollary}
  Let \eqref{e5.1} be satisfied, let $k_1,\hdots,k_m\in\N$, $m\in\N$, and let
  $\Sigma=(\sigma_{\varphi,\varphi}^{(k_i,k_j)})_{i,j\in[m]}$
  be as in Corollary \ref{thm:AsymptoticVariance2}. Then
$$
\frac{1}{\sqrt{\lambda_d(W)}} \big(\eta_{k_1}(W)-\BE
\eta_{k_1}(W),\hdots,\eta_{k_m}(W)-\BE
\eta_{k_m}(W)\big) \overset{d}{\longrightarrow} N_\Sigma \quad \text{as} \quad r(W) \to\infty.
$$
\end{corollary}

In \cite[Theorem 1.1 and Theorem 4.1]{BrugMeester04} univariate central limit theorems for the number of $k$-components, $k\in\N$, were derived in case of a monotone and isotropic connection function that has also bounded support if $k\geq 2$. For the number of isolated vertices a central limit theorem with an erroneous proof was given in \cite{RoySarkar2003} (see the discussion in \cite{BrugMeester04}).

Theorem \ref{thm:CLTSubgraphs} will be deduced from the following univariate central limit theorem for the random variables $S_{a,G}(W)$ with $G=(G_1,\ldots,G_m)\in \mathbf{G}_\varphi^{m,\ne}$, $m\in\N$, $a=(a_1,\ldots,a_m)\in\R^m$ and $W\in\mathcal{K}^d$ introduced in \eqref{eqn:functional.S.a.G}. Recall the definition of $\mathbf{G}_\varphi^{m,\ne}$ given in \eqref{eqn:set.G.ne} and that $N$ denotes a standard normal random variable.

\begin{theorem}\label{thm:CLTGeneral}
Assume that \eqref{e5.1} is satisfied, let $m\in\N$, $G=(G_1,\hdots,G_m)\in\mathbf{G}_\varphi^{m,\ne}$ and $a=(a_1,\hdots,a_m)\in\R^m$ with $a\neq 0$. Then
$$
\frac{S_{a,G}(W) -\BE S_{a,G}(W)}{\sqrt{\V S_{a,G}(W)}} \longrightarrow N \quad \text{as} \quad r(W)\to\infty,
$$
where the convergence holds in the $d_1$-distance and, in particular, in distribution.
\end{theorem}

Under a slightly stronger integrability condition on $\varphi$ than \eqref{e5.1} we are even able to derive quantitative univariate central limit theorems. In the following we assume that there exists a function $\tilde{\varphi}: [0,\infty)\to[0,1]$ such that
\begin{equation}\label{eqn:TildeVarphi}
\varphi(x)\le \tilde\varphi(|x|),\quad x\in\R^d, \quad \tilde{\varphi}(s)\geq \tilde{\varphi}(t), \quad 0\leq s \leq t, \quad \text{and} \quad
\int_{\R^d} \tilde\varphi(|x|)^{1/3} dx < \infty.
\end{equation}
Note that the last condition implies the upper bound in \eqref{e5.1}.

\begin{theorem}\label{thm:QuantitativeGeneral}
  Assume $m_\varphi>0$ and \eqref{eqn:TildeVarphi}. Then for any $m\in\N$,
  $G=(G_1,\hdots,G_m)\in\mathbf{G}_\varphi^{m,\ne}$
  and $a=(a_1,\hdots,a_m)\in\R^m$ with $a\neq 0$ there exist constants
  $C, \tau>0$ only depending on $\beta$, $\varphi$, $\tilde{\varphi}$, $G$ and $a$ such that
$$
d_K\bigg(\frac{S_{a,G}(W) -\BE S_{a,G}(W)}{\sqrt{\V
    S_{a,G}(W)}},N \bigg) \leq \frac{C}{\sqrt{\lambda_d(W)}}
$$
for all $W\in\mathcal{K}^d$ with $r(W)\geq \tau$.
\end{theorem}

For component counts Theorem \ref{thm:QuantitativeGeneral} leads to the following bounds for the
Kolmogorov distance.

\begin{corollary}\label{cor:QuantitativeSupgraph}
  Assume $m_\varphi>0$ and \eqref{eqn:TildeVarphi}.
\begin{itemize}
\item [(a)] For any $G\in\mathbf{G}_\varphi$ there exist constants $C,\tau>0$ only depending on $\beta$, $\varphi$, $\tilde{\varphi}$ and $G$ such that
$$
d_K\bigg(\frac{\eta_{G}(W)-\BE\eta_{G}(W)}{\sqrt{\V
    \eta_{G}(W)}},N\bigg) \leq \frac{C}{\sqrt{\lambda_d(W)}}
$$
for all $W\in\mathcal{K}^d$ with $r(W)\geq \tau$.
\item [(b)] For any $k\in\N$ there exist constants $C,\tau>0$ only depending on $\beta$, $\varphi$, $\tilde{\varphi}$ and $k$ such that
$$
d_K\bigg(\frac{\eta_{k}(W)-\BE\eta_{k}(W)}{\sqrt{\V
    \eta_{k}(W)}},N\bigg) \leq \frac{C}{\sqrt{\lambda_d(W)}}
$$
for all $W\in\mathcal{K}^d$ with $r(W)\geq \tau$.
\end{itemize}
\end{corollary}

One can also show quantitative bounds for the Wasserstein distance.

\begin{theorem}\label{thm:CLTWasserstein}
  The assertions of Theorem \ref{thm:QuantitativeGeneral} and Corollary
  \ref{cor:QuantitativeSupgraph} also hold for the Wasserstein
  distance $d_1$.
\end{theorem}

The proofs of the findings presented in this section are organized as follows: After deriving Theorem \ref{thm:CLTSubgraphs} from Theorem \ref{thm:CLTGeneral}, Theorem \ref{thm:CLTGeneral} is proven by Theorem \ref{thm:CLTWasserstein}. Thereafter, the quantitative bounds in Theorems \ref{thm:QuantitativeGeneral} and \ref{thm:CLTWasserstein} are established by applying Theorem \ref{thm:General.CLT}.

\begin{proof}[Proof of Theorem \ref{thm:CLTSubgraphs}]
It follows from Theorem \ref{thm:VarianceS} and Theorem
\ref{thm:CLTGeneral} that for any $a=(a_1,\hdots,a_m)\in\R^m$ with $a\neq 0$,
$$
\frac{\sum_{i=1}^m a_i (\eta_{G_i}(W) - \BE
  \eta_{G_i}(W))}{\sqrt{\lambda_d(W)}}
\overset{d}{\longrightarrow} a^{T} N_\Sigma a
$$
as $r(W)\to\infty$. Since this is obviously true for $a=0$, the
Cramer-Wold theorem yields the assertion.
\end{proof}

We prepare the proof of Theorem \ref{thm:CLTGeneral} by the following lemma.

\begin{lemma}\label{lem:L2Approximation}
  Let \eqref{e5.1} be satisfied and let $(\psi_n)_{n\in\N}$ be a family of connection functions such
  that $\psi_n\leq\varphi$ for any $n\in\N$ and
  $\lim_{n\to\infty}\psi_n(x)=\varphi(x)$ for $x\in\R^d$. Then, for all $G,H\in\mathbf{G}$,
$$
\lim_{n\to\infty} \sigma_{\varphi,\psi_n}(G,H) =
\sigma_{\varphi,\varphi}(G,H) \quad \text{ and } \quad
\lim_{n\to\infty} \sigma_{\psi_n,\psi_n}(G,H) =
\sigma_{\varphi,\varphi}(G,H).
$$
\end{lemma}

\begin{proof}
Note that the integrands in the representations of
$\sigma_{\varphi,\psi_n}(G,H)$ and $\sigma_{\psi_n,\psi_n}(G,H)$ given in Theorem \ref{thm:AsymptoticVariance}
are dominated by the integrable functions only depending on $\varphi$ that are derived in the proof of Theorem \ref{thm:AsymptoticVariance}. Due to the pointwise convergence of
$(\psi_n)_{n\in\N}$ to $\varphi$, the integrands also converge
pointwise to the integrands of $\sigma_{\varphi,\varphi}(G,H)$. Thus,
the dominated convergence theorem completes the proof.
\end{proof}

\begin{proof}[Proof of Theorem \ref{thm:CLTGeneral}]
For $n\in\N$
let the connection function $\psi_n:\R^d\to[0,1]$ be given by $\psi_n(x):={\bf
  1}\{|x|\leq n\}\varphi(x)$. Let $n_0\in \N$ be such that $m_{\psi_{n_0}}>0$ and, thus, $m_{\psi_n}>0$ for $n\geq n_0$.
Throughout this proof we add the connection function as a further index to $S_{a,G}(W)$ and use the abbreviation
$$
\widehat{S}_{\chi,a,G}(W) := \frac{S_{\chi,a,G}(W)-\BE S_{\chi,a,G}(W)}{\sqrt{\V S_{\chi,a,G}(W)}}
$$
for $\chi=\varphi$ or $\chi=\psi_n$, $n\geq n_0$.

By the triangle inequality and the
fact that the Wasserstein distance can be bounded by the $L^1$-distance and the $L^2$-distance we obtain that, for $n\geq n_0$,
\begin{align*}
  & d_1\big(\widehat{S}_{\varphi,a,G}(W) ,N\big)\\
  & \leq d_1\bigg( \widehat{S}_{\varphi,a,G}(W),\frac{S_{\psi_n,a,G}(W)-\BE S_{\psi_n,a,G}(W)}{\sqrt{\V S_{\varphi,a,G}(W)}}\bigg)\\
  & \quad + d_1\bigg( \frac{S_{\psi_n,a,G}(W)-\BE S_{\psi_n,a,G}(W)}{\sqrt{\V S_{\varphi,a,G}(W)}},\widehat{S}_{\psi_n,a,G}(W)\bigg)+ d_1\big( \widehat{S}_{\psi_n,a,G}(W),N\big)\\
  & \leq \frac{\sqrt{\V \big(S_{\varphi,a,G}(W)-S_{\psi_n,a,G}(W)\big)}}{\sqrt{\V S_{\varphi,a,G}(W)}} + \bigg|\frac{\sqrt{\V S_{\psi_n,a,G}(W)}}{\sqrt{\V S_{\varphi,a,G}(W)}}-1\bigg|
  + d_1\big( \widehat{S}_{\psi_n,a,G}(W),N\big).
\end{align*}
Since the assumptions of Theorem \ref{thm:CLTWasserstein} are satisfied for $n\geq n_0$, the last term on the right-hand side vanishes as
$r(W)\to\infty$. Consequently, Theorem \ref{thm:AsymptoticVariance}
implies that
\begin{align*}
\limsup_{r(W)\to\infty} d_1\big(\widehat{S}_{\varphi,a,G}(W),N\big)
  & \leq \sqrt{\frac{\sum_{i,j=1}^m a_i a_j (\sigma_{\varphi,\varphi}(G_i,G_j)+\sigma_{\psi_n,\psi_n}(G_i,G_j)-2\sigma_{\varphi,\psi_n}(G_i,G_j))}{\sum_{i,j=1}^m a_i a_j\sigma_{\varphi,\varphi}(G_j,G_j)}}\\
  & \quad +\bigg| \sqrt{\frac{\sum_{i,j=1}^m
      a_i a_j\sigma_{\psi_n,\psi_n}(G_i,G_j)}{\sum_{i,j=1}^m a_i a_j
      \sigma_{\varphi,\varphi}(G_i,G_j)}}-1 \bigg|
\end{align*}
for $n\geq n_0$. Now taking the limit $n\to\infty$ and Lemma \ref{lem:L2Approximation} yield
$$
\lim_{r(W)\to\infty} d_1\big( \widehat{S}_{\varphi,a,G}(W),N\big)=0,
$$
which concludes the proof.
\end{proof}

The rest of this section is devoted to the proofs of Theorem \ref{thm:QuantitativeGeneral} and Theorem \ref{thm:CLTWasserstein}, which are based on the following three lemmas. For $n \in \N$ and $x_1,\ldots, x_n\in\R^d$ we define
$\eta_{x_1,\ldots,x_n} := \eta\cup\{x_1,\ldots,x_n\}$. We refer to Section \ref{secpre} for further notation.

\begin{lemma}\label{lem.first.diff}
Let the assumptions of Theorem \ref{thm:QuantitativeGeneral} prevail and let $k := \max \{|G_1|,\ldots,|G_m|\}$.
Then, for any measurable set $W\subset\R^d$ and $x, y \in \R^d$,
\begin{equation}\label{eqn:bound.first.diff.op}
|\Delta_x S_{a,G}(W)|\le |a |_\infty (\deg(x,\Gamma(\eta_x))+1)\1\big\{x\overset{\le k}{\llra}W \text{ in }\Gamma(\eta_x)\big\}
\end{equation}
and
\begin{align}\label{eqn:bound.sec.diff.op}
|\Delta^2_{x,y} S_{a,G}(W)| &\le | a |_\infty (2\deg(y,\Gamma(\eta_y))+3)\1\big\{x\overset{\le k+1}{\llra}y \text{ in } \Gamma(\eta_{x,y})\big\}\nonumber\\
&\qquad\times\1\big\{x\overset{\le k}{\llra}W \text{ in } \Gamma(\eta_x) \text{ or } y\overset{\le k}{\llra}W \text{ in } \Gamma(\eta_y)\big\}.
\end{align}
\end{lemma}

\begin{proof}
We use the abbreviation $F:=S_{a,G}(W)$. If $\Delta_x F\ne 0$, the component of $\Gamma(\eta_x)$ containing $x$ forms a counted copy of one of the graphs $G_1,\ldots,G_m$, or $x$ is connected with at least one component of $\Gamma(\eta)$ that contributes to $F$.
In both cases we have $x\overset{\le k}{\llra}W$ in $\Gamma(\eta_x)$.
The number of counted components in $\Gamma(\eta)$ that are connected with $x$ by an edge in $\Gamma(\eta_x)$ is bounded by $\deg(x,\Gamma(\eta_x))$.
Since the addition of $x$ can also create one new counted component, we obtain \eqref{eqn:bound.first.diff.op}.

For the second difference operator we have the representation
\begin{align}\label{eqn:second.diff.op.alternative}
\Delta^2_{x,y} F=\sum_{\emptyset\neq \mu\subset\eta_{x,y}}h_{\{x,y\}}(\mu)-h_{\{x\}}(\mu)-h_{\{y\}}(\mu)+h(\mu),
\end{align}
where for $A\subset \{x,y\}$, $h_A(\mu):=a_i$ if $\mu\subset \eta\cup A$, the lexicographic minimum
of $\mu$ is in $W$, and $\mu$ is isomorphic to $G_i$ in $\Gamma(\eta\cup A)$ for some $i\in[m]$ and $h_A(\mu):=0$, otherwise, and $h(\mu):=h_{\emptyset}(\mu)$.
Let $\mu\subset \eta_{x,y}$ be finite and non-empty and assume that neither $x$ nor $y$
is connected with $\mu$
in $\Gamma(\eta_{x,y})$. Then $h_{\{x,y\}}(\mu)=h_{\{x\}}(\mu)=h_{\{y\}}(\mu)=h(\mu)$ and
$\mu$ does not contribute to \eqref{eqn:second.diff.op.alternative}. Assume next that $x$ is connected with $\mu$
in $\Gamma(\eta_{x,y})$ but $y$ is not. Then $h_{\{x,y\}}(\mu)=h_{\{x\}}(\mu)$ and
$h_{\{y\}}(\mu)=h(\mu)$ so that $\mu$ does again not contribute to \eqref{eqn:second.diff.op.alternative}.
This shows that $\Delta^2_{x,y} F=0$ unless there is a component of $\Gamma(\eta_{x,y})$
that
contains both $x$ and $y$ and where $x$ and $y$ are connected via at most $k+1$ edges (otherwise the component would be too large). Set $\Delta_{y}F_{x} := f(\xi_{x,y})-f(\xi_{x})$,
where $f$ is a representative of $F$. Noting that
$\Delta^2_{x,y} F=\Delta_{y}F_{x}-\Delta_{y}F$ it now follows that
\begin{align*}
\big|\Delta^2_{x,y} F\big|
\le \1\big\{x \overset{\le k+1}{\llra} y \text{ in } \Gamma(\eta_{x,y})\big\}
(|\Delta_{y}F_{x}| + |\Delta_{y}F|).
\end{align*}
We use \eqref{eqn:bound.first.diff.op} to bound $|\Delta_y F|$ and analogously $|\Delta_y F_x|$, which leads to
\begin{align*}
|\Delta^2_{x,y} F| &\le |a|_\infty \1\big\{x \overset{\le k+1}{\llra} y \text{ in } \Gamma(\eta_{x,y})\big\}\\
&\quad\times\big((\deg(y,\Gamma(\eta_{x,y}) +1)\1\big\{y \overset{\le k}{\llra} W \text{ in } \Gamma(\eta_{x,y})\big\}\\
&\qquad + (\deg(y,\Gamma(\eta_y)+1)\1\big\{y \overset{\le k}{\llra} W \text{ in } \Gamma(\eta_y)\big\}\big)\\
&\le |a|_\infty \1\big\{x \overset{\le k+1}{\llra} y \text{ in } \Gamma(\eta_{x,y})\big\}\1\big\{y \overset{\le k}{\llra} W \text{ in } \Gamma(\eta_{x,y})\big\} (2\deg(y,\Gamma(\eta_y))+3)\\
&\le |a|_\infty(2\deg(y,\Gamma(\eta_y))+3) \1\big\{x \overset{\le k+1}{\llra} y \text{ in } \Gamma(\eta_{x,y})\big\}\\
&\quad\times \1\big\{x\overset{\le k}{\llra}W \text{ in } \Gamma(\eta_x) \text{ or } y\overset{\le k}{\llra}W \text{ in } \Gamma(\eta_y)\big\}.
\end{align*}
This finishes the proof of \eqref{eqn:bound.sec.diff.op}.
\end{proof}

\begin{lemma}\label{lem.diff.bound}
Let the assumptions of Theorem \ref{thm:QuantitativeGeneral} prevail, let $W\subset\R^d$ be a compact set and $k:=\max\{|G_1|,\ldots,|G_m|\}$.
Then there are constants $C_1,C_2,C_3,C_4,C_5>0$ only depending on $\beta$, $\tilde\varphi$, $k$ and $|a|_\infty$ such that, for $x,y,z\in\R^d$,
\begin{align}
\BE(\Delta_x S_{a,G}(W))^4 &\leq C_1 \tilde{\varphi}(d(x,W)/k)^{2/3},\label{eqn:exp.diff.1} \allowdisplaybreaks\\
\BE(\Delta_xS_{a,G}(W))^2(\Delta_y S_{a,G}(W))^2 &\leq C_2\big( \tilde{\varphi}(d(x,W)/k)^{2/3} + \tilde{\varphi}(d(y,W)/k)^{2/3} \big), \label{eqn:exp.diff.2} \allowdisplaybreaks\\
\BE \big(\Delta^2_{x,y}S_{a,G}(W)\big)^4 &\le C_3 \big( \tilde{\varphi}(d(x,W)/(2k))^{2/3} + \tilde{\varphi}(d(y,W)/(2k))^{2/3} \big) \nonumber\\
&\quad\times\tilde{\varphi}(|x-y|/(k+1))^{2/3},\label{eqn:exp.diff.3} \allowdisplaybreaks\\
\BE\big(\Delta^2_{x,z}S_{a,G}(W)\big)^2(\Delta^2_{y,z}S_{a,G}(W))^2 &\leq C_4 \big( \tilde{\varphi}(d(x,W)/(2k))^{2/3} + \tilde{\varphi}(d(y,W)/(2k))^{2/3}\nonumber\\
&\quad \quad\quad + \tilde{\varphi}(d(z,W)/(2k))^{2/3} \big),\label{eqn:exp.diff.4}
\end{align}
and
\begin{align}\label{eqn:exp.diff.5}
\BE &\big(\Delta^2_{x,z}S_{a,G}(W)\big)^2 \big(\Delta^2_{y,z}S_{a,G}(W)\big)^2\nonumber\\
&\leq C_5 \tilde{\varphi}(|x-z|/(2k+1))^{2/3} \big( \tilde{\varphi}(|x-y|/(2k+1))^{2/3} + \tilde{\varphi}(|y-z|/(2k+1))^{2/3} \big).
\end{align}
\end{lemma}

\begin{proof}
For $n \in \N$ and $v_1,\ldots,v_n \in \R^d$ we define
$$
\Phi_n(v_1,\ldots,v_n) :=
\prod_{i=1}^{n-1} \tilde\varphi(|v_i-v_{i+1}|),
$$
where the empty product equals one. For $A \subset \R^d$ let
$$
\widetilde{\Phi}_n(v_1,\ldots,v_n;A) := \Phi_n(v_1,\ldots,v_n) \1\{v_n \in A\}.
$$
If $v_1,\ldots,v_n$, $n\in\N$, are vertices of a given graph $H$, we define
$$
\Theta_n(v_1,\ldots,v_n,H) := \prod_{i=1}^{n-1} \1\{v_i \lra v_{i+1} \text{ in } H\}
$$
and
$$
\widetilde{\Theta}_n(v_1,\ldots,v_n,H;A) := \Theta_n(v_1,\ldots,v_n,H)\1\{v_n \in A\}
$$
for $A\subset\R^d$. For $j=0$ we use the conventions
$$\int 1 \, d(w_1,\ldots,w_j) := 1 \qquad \text{and} \qquad \sideset{}{^{\ne}}\sum_{(w_1,\ldots,w_j)\in\eta^j} 1 := 1.$$

In the following let $x,y,z\in\R^d$ and $F:=S_{a,G}(W)$. Without loss of generality we can assume that $|a|_\infty =1$.\\
Lemma $\ref{lem.first.diff}$ implies that
\begin{align*}
\BE (\Delta_xF)^4 &\le \BE \big[(\deg(x,\Gamma(\eta_x))+1)^4 \1\big\{x\overset{\le k}{\llra}W \text{ in } \Gamma(\eta_x)\big\}\big].
\end{align*}
The multivariate Mecke equation \eqref{Mecke} yields that the right-hand side of the previous inequality can be bounded by
\begin{align*}
&\sum_{i=0}^k \beta^i \int \BE (\deg(x,\Gamma(\eta_{x,w_1,\ldots,w_i}))+1)^4 \, \widetilde{\Theta}_{i+1}(x,w_1,\ldots,w_i,\Gamma(\eta_{x,w_1,\ldots,w_i});W) \, d(w_1,\ldots,w_i)\\
&\leq \sum_{i=0}^k \beta^i \int \BE (\deg(x,\Gamma(\eta_x))+i+1)^4 \, \widetilde{\Theta}_{i+1}(x,w_1,\ldots,w_i,\Gamma(\eta_{x,w_1,\ldots,w_i});W) \, d(w_1,\ldots,w_i).
\end{align*}
Since for given $i\in[k]_0$ and $w_1,\ldots,w_i\in\R^d$ the variables $\deg(x,\Gamma(\eta_x))$
and $\widetilde{\Theta}_{i+1}(x,w_1,\ldots,w_i,\Gamma(\eta_{x,w_1,\ldots,w_i});W)$ are independent and $\varphi\leq \tilde{\varphi}$,
we obtain that the right-hand side of the previous inequality can be bounded by
\begin{equation}\label{eq.bound.diff.op.1}
\widetilde C_1 \sum_{i=0}^k \beta^i \int \widetilde{\Phi}_{i+1}(x,w_1,\ldots,w_i;W)\, d(w_1,\ldots,w_i)
\end{equation}
with $\widetilde C_1 := \BE (\deg(x,\Gamma(\eta_x))+k+1)^4$.

For $i \in [k]$ and $w_1,\ldots,w_i\in\R^d$ with $w_i\in W$ we have
$$
\max\{|x-w_1|,|w_1-w_2|,\ldots,|w_{i-1}-w_i|\} \geq d(x,W)/i \ge d(x,W)/k.
$$
Since $\tilde{\varphi}$ is decreasing, this shows
$$
\Phi_{i+1}(x,w_1,\ldots,w_i) = \tilde\varphi(|x - w_1|) \tilde{\varphi}(|w_1-w_2|) \cdots \tilde\varphi(|w_{i-1} - w_i|) \le \tilde\varphi(d(x,W)/k).
$$
Hence, (\ref{eq.bound.diff.op.1}) can be bounded by
$$
\widetilde{C}_1 \tilde\varphi(d(x,W)/k)^{2/3} \sum_{i=0}^k \beta^i \int \Phi_{i+1}(x,w_1,\ldots,w_i)^{1/3}\, d(w_1,\ldots,w_i)
\le C_1 \tilde\varphi(d(x,W)/k)^{2/3}
$$
with a constant $C_1>0$, where we have used the final part of \eqref{eqn:TildeVarphi}.
This proves (\ref{eqn:exp.diff.1}).

Inequality \eqref{eqn:exp.diff.2} can be shown by using the fact that
$$
\BE (\Delta_x F)^2 (\Delta_y F)^2\leq \BE(\Delta_x F)^4 + \BE(\Delta_y F)^4
$$
and by applying inequality \eqref{eqn:exp.diff.1}.

We now turn to the proof of \eqref{eqn:exp.diff.3}. Lemma \ref{lem.first.diff} yields
\begin{align}\label{eq:lem.diff.op.third.1}
\BE \big(\Delta^2_{x,y}F\big)^4 &\le \BE \big[ \left(2\deg(y,\Gamma(\eta_y))+3\right)^4 \1\big\{x\overset{\le k+1}{\llra}y \text{ in } \Gamma(\eta_{x,y})\big\}\nonumber\\
&\quad \times \big( \1\big\{x\overset{\le k}{\llra}W \text{ in } \Gamma(\eta_x)\big\} + \1\big\{y\overset{\le k}{\llra}W \text{ in } \Gamma(\eta_y)\big\} \big) \big].
\end{align}
Considering the event
$$
\big\{x\overset{\le k+1}{\llra}y \text{ in } \Gamma(\eta_{x,y}), \, x\overset{\le k}{\llra}W \text{ in } \Gamma(\eta_x)\big\}
$$
we have to distinguish two cases. Either the path connecting $x$ and $y$ in $\Gamma(\eta_{x,y})$ is disjoint from the path connecting $x$ with $W$ in $\Gamma(\eta_x)$, or the two paths share at least one common vertex except $x$.
This leads to
\begin{align*}
&\1 \big\{x\overset{\le k+1}{\llra}y \text{ in } \Gamma(\eta_{x,y})\big\}\1\big\{x\overset{\le k}{\llra}W \text{ in } \Gamma(\eta_x)\big\}\\
&\le \sum_{i=0}^k \sum_{j=0}^k \; \sideset{}{^{\ne}}\sum_{(v_1,\ldots,v_i,w_1,\ldots,w_j)\in\eta^{i+j}} \Theta_{i+2}(x,v_1,\ldots,v_i,y,\Gamma(\eta_{x,y}))\widetilde{\Theta}_{j+1}(x,w_1,\ldots,w_j,\Gamma(\eta_x);W) \\
&+ \sum_{i=1}^k \sum_{l=1}^i \sum_{j=0}^k \; \sideset{}{^{\ne}}\sum_{(v_1,\ldots,v_i,w_1,\ldots,w_j)\in\eta^{i+j}} \Theta_{i+2}(x,v_1,\ldots,v_i,y,\Gamma(\eta_{x,y}))\widetilde{\Theta}_{j+1}(v_l,w_1,\ldots,w_j,\Gamma(\eta);W).
\end{align*}
Similarly to the proof of the first inequality we have for $i,j\in[k]_0$,
\begin{align}\label{eq:lem.diff.op.third.1.5}
\BE &\sideset{}{^{\ne}}\sum_{(v_1,\ldots,v_i,w_1,\ldots,w_j)\in\eta^{i+j}} (2\deg(y,\Gamma(\eta_y))+3)^4\,
 \Theta_{i+2}(x,v_1,\ldots,v_i,y,\Gamma(\eta_{x,y})\nonumber\\
 &\hspace{8cm}\times \widetilde{\Theta}_{j+1}(x,w_1,\ldots,w_j,\Gamma(\eta_x);W) \nonumber\\
&\leq \beta^{i+j} \int \BE (2\deg(y,\Gamma(\eta_y))+i+j+3)^4\, \Phi_{i+2}(x,v_1,\ldots,v_i,y)\nonumber\\
&\hspace{6cm}\times\widetilde{\Phi}_{j+1}(x,w_1,\ldots,w_j;W)\, d(v_1,\ldots,v_i,w_1,\ldots,w_j)\nonumber\\
&\le \widetilde{C}_3 \tilde{\varphi}(|x-y|/(k+1))^{2/3} \tilde{\varphi}(d(x,W)/k)^{2/3}
\end{align}
with a constant $\widetilde{C}_3>0$. Analogously we have for $i \in[k]$, $l\in[i]$ and $j\in[k]_0$,
\begin{align}\label{eq:lem.diff.op.third.2}
\BE & \sideset{}{^{\ne}}\sum_{(v_1,\ldots,v_i,w_1,\ldots,w_j)\in\eta^{i+j}} (2\deg(y,\Gamma(\eta_y))+3)^4\, \Theta_{i+2}(x,v_1,\ldots,v_i,y,\Gamma(\eta_{x,y}))\nonumber\\
&\hspace{8cm}\times\widetilde{\Theta}_{j+1}(v_l,w_1,\ldots,w_j,\Gamma(\eta);W)\nonumber\\
&\leq \beta^{i+j} \int \BE (2\deg(y,\Gamma(\eta_y))+i+j+3)^4\, \Phi_{i+2}(x,v_1,\ldots,v_i,y)\nonumber\\
&\hspace{5cm}\times\widetilde{\Phi}_{j+1}(v_l,w_1,\ldots,w_j;W) \, d(v_1,\ldots,v_i,w_1,\ldots,w_j).
\end{align}
In case
$$
\max\{|x-v_1|, |v_1-v_2|, \ldots, |v_{l-1}-v_l|\} \ge \max\{|v_l-v_{l+1}|, \ldots, |v_{i-1}-v_i|, |v_i-y|\}
$$
we use the inequalities
\begin{align*}
\Phi_{l+1}(x,v_1,\ldots,v_l) &\leq \tilde\varphi(|x-y|/(k+1)),\\
\Phi_{i-l+2}(v_l,\ldots,v_i,y) \widetilde{\Phi}_{j+1}(v_l,w_1,\ldots,w_j) &\leq \tilde\varphi(d(y,W)/(2k))
\end{align*}
to bound the integrand of (\ref{eq:lem.diff.op.third.2}). Analogously, if
$$
\max\{|x-v_1|, |v_1-v_2|, \ldots, |v_{l-1}-v_l|\} < \max\{|v_l-v_{l+1}|, \ldots, |v_{i-1}-v_i|, |v_i-y|\}
$$
the inequalities
\begin{align*}
\Phi_{i-l+2}(v_l,\ldots,v_i,y) &\leq \tilde\varphi(|x-y|/(k+1)),\\
\Phi_{l+1}(x,v_1,\ldots,v_l) \widetilde{\Phi}_{j+1}(v_l,w_1,\ldots,w_j;W) &\leq \tilde\varphi(d(x,W)/(2k))
\end{align*}
are used.
In summary, similarly to \eqref{eq:lem.diff.op.third.1.5} we obtain that \eqref{eq:lem.diff.op.third.2} can be bounded by
$$
\widehat{C}_3 \tilde\varphi(|x-y|/(k+1))^{2/3}\big( \tilde\varphi(d(x,W)/(2k))^{2/3} + \tilde\varphi(d(y,W)/(2k))^{2/3} \big)
$$
with a constant $\widehat{C}_3>0$. The same arguments hold for
$$
\BE (2\deg(y,\Gamma(\eta_y))+3)^4 \1 \big\{x\overset{\le k+1}{\llra}y \text{ in } \Gamma(\eta_{x,y})\big\}\1\big\{y\overset{\le k}{\llra}W \text{ in } \Gamma(\eta_y)\big\},
$$
so that \eqref{eq:lem.diff.op.third.1} implies \eqref{eqn:exp.diff.3}.

Analogously to \eqref{eqn:exp.diff.2}, we obtain \eqref{eqn:exp.diff.4} by using
$$
\BE \big(\Delta^2_{x,z}F\big)^2\big(\Delta^2_{y,z}F\big)^2 \leq \BE \big(\Delta^2_{x,z}F\big)^4 + \BE \big(\Delta^2_{y,z}F\big)^4
$$
and applying inequality \eqref{eqn:exp.diff.3}. The inequality \eqref{eqn:exp.diff.5} can be proven with similar arguments as \eqref{eqn:exp.diff.3}. This is left to the reader.
\end{proof}

\begin{lemma}\label{lem.min}
Let $\alpha>0$ and let ${\chi}: [0,\infty)\to[0,1]$ be a monotonously decreasing function such that $\int_{\R^d} {\chi}(|x|)^{\alpha}\, dx<\infty$.
Then there exists a monotonously decreasing function $h:(0,\infty)\to[0,\infty)$
with $h(t)\to 0$ as $t\to\infty$, such that
\begin{equation*}
\frac{1}{\lambda_d(W)} \int{\chi}(d(x,W))^\alpha \,dx \le 1 + h(r(W)), \quad W\in\mathcal{K}^d.
\end{equation*}
\end{lemma}

\begin{proof}
For $W\in\mathcal{K}^d$ the local Steiner formula in \cite[Theorem 4.2.8]{Schneider13} yields
\begin{align*}
\int \chi (d(x,W))^\alpha  d x & \leq \lambda_d (W) + \int_{\R^d \setminus W} \chi (d(x,W))^\alpha  d x \\
&= \lambda_d (W) +  \sum_{j=0}^{d-1} (d-j) \kappa_{d-j} V_j (W) \int_0^\infty t^{d-j-1} \chi (t)^\alpha  d t,
\end{align*}
where $V_0, \dotsc, V_{d-1}$ denote the intrinsic volumes and $\kappa_j$ stands for the volume of the $j$-dimensional unit ball, $j\in[d]_0$.
We have that
\begin{align*}
S(W)&:=\sum_{j=0}^{d-1} (d-j) \kappa_{d-j} V_j (W) \int_0^\infty t^{d-j-1} \chi (t)^\alpha  d t \\
&\le \sum_{j=0}^{d-1} (d-j) \kappa_{d-j} V_j (W) \bigg( \chi (0)^\alpha + \int_1^\infty t^{d-1} \chi(t)^\alpha  d t \bigg) \\
&\le \sum_{j=0}^{d-1} (d-j) \kappa_{d-j} V_j (W) \bigg( 1 + \frac{1}{d \kappa_d} \int_{\R^d} \chi (|x|)^\alpha  d x \bigg).
\end{align*}
Now \cite[Lemma 3.7]{HugLastSchulte2016} yields for $j\in[d-1]_0$,
$$
\frac{V_j(W)}{\lambda_d(W)} \le \frac{2^d-1}{\kappa_{d-j} r(W)^{d-j}}.
$$
Hence, we obtain
\begin{equation*}
\frac{S(W)}{\lambda_d(W)} \le \sum_{j=0}^{d-1} \frac{(d-j)(2^d-1)}{r(W)^{d-j}} \bigg( 1 + \frac{1}{d \kappa_d} \int_{\R^d} \chi (| x|)^\alpha  d x \bigg) \longrightarrow 0 \quad \text{as} \quad r(W)\to\infty.
\end{equation*}
This finishes the proof of the lemma.
\end{proof}

\begin{proof}[Proof of Theorem \ref{thm:QuantitativeGeneral} and Theorem \ref{thm:CLTWasserstein}]
Our aim is to apply Theorem \ref{thm:General.CLT}.
Let $k:=\max\{|G_1|,\ldots,|G_m|\}$.

By Corollary \ref{cor:lowerBoundVar} there are constants $\tau, c >0$ such that
\begin{equation}\label{eqn:var.low.bound}
\V S_{a,G}(W) \geq c \lambda_d(W)
\end{equation}
for all $W\in\mathcal{K}^d$ with $r(W)\geq\tau$.

It follows from \eqref{eqn:TildeVarphi} and Lemma \ref{lem.min} that there exists a constant $\tilde{c}>0$ such that
\begin{equation}\label{eqn:bound.int.con.3root}
\int \tilde\varphi(d(x,W)/l)^{1/3} \,dx \le \tilde{c} \lambda_d(W)
\end{equation}
for all $W\in\mathcal{K}^d$ with $r(W)\ge\tau$ and $l\in\{k,\ldots,2k+1\}$.

From now on let $W\in\mathcal{K}^d$ with $r(W) \geq \tau$ and define $\widetilde{F} := S_{a,G}(W)$ and
$$
F:= \frac{S_{a,G}(W) - \BE S_{a,G}(W)}{\sqrt{\V S_{a,G}(W)}}.
$$
Let the quantities $\gamma_1,\ldots,\gamma_6$ be defined as in Section \ref{sec:normal.approx} with respect to $F$. It follows from the obvious inequality $|\widetilde{F}| \le |a|_\infty \eta(W)$ that all moments of $\widetilde{F}$ and $F$ exist and, in particular, $\BE F^4<\infty$.

Let $C_1,\ldots,C_5$ be the constants from Lemma \ref{lem.diff.bound}. The Cauchy-Schwarz inequality, \eqref{eqn:exp.diff.1} and \eqref{eqn:exp.diff.3} yield
\begin{align}\label{eqn:proof.quantitative.clt.1}
&\int \Big[\BE \big(\Delta_{x_1}\widetilde{F}\big)^4\Big]^{1/4} \Big[\BE \big(\Delta_{x_2}\widetilde{F}\big)^4\Big]^{1/4}
\Big[\BE \big(\Delta_{x_1,x_3}^2\widetilde{F}\big)^4\Big]^{1/4}\Big[\BE  \big(\Delta_{x_2,x_3}^2\widetilde{F}\big)^4\Big]^{1/4}\, d(x_1,x_2,x_3)\nonumber\\
&= \int \biggl[ \int \Big[\BE\big(\Delta_y \widetilde{F}\big)^4\Big]^{1/4} \Big[\BE\big(\Delta_{y,x}^2\widetilde{F}\big)^4\Big]^{1/4} dy\biggr]^2dx\nonumber\\
&\le \iint \Big[\BE \big(\Delta_{z}\widetilde{F}\big)^4\Big]^{1/2} dz \int \Big[\BE\big(\Delta_{y,x}^2 \widetilde{F}\big)^4\Big]^{1/2} dy \, dx\nonumber\\
&\le \sqrt{C_1C_3} \int \tilde{\varphi}(d(z,W)/k)^{1/3} dz\nonumber\\
&\quad \times \int \tilde{\varphi}(|x-y|/(k+1))^{1/3}\big[ \tilde{\varphi}(d(x,W)/(2k))^{2/3} + \tilde{\varphi}(d(y,W)/(2k))^{2/3} \big]^{1/2} d(x,y).
\end{align}
We apply the inequality $\sqrt{b_1+b_2}\leq \sqrt{b_1}+\sqrt{b_2}$, $b_1,b_2\geq 0$, and Lemma \ref{lem.min} together with \eqref{eqn:TildeVarphi} and obtain that the right-hand side of \eqref{eqn:proof.quantitative.clt.1} is finite.

Inequality (\ref{eqn:exp.diff.3}) yields
\begin{align*}
&\int \Big[\BE \big(\Delta_{x_1,x_3}^2\widetilde{F}\big)^4\Big]^{1/2} \Big[\BE\big(\Delta_{x_2,x_3}^2\widetilde{F}\big)^4\Big]^{1/2} \,d(x_1,x_2,x_3)\\
&\le C_3 \int \tilde{\varphi}(|x_1-x_3|/(k+1))^{1/3}\big[\tilde{\varphi}(d(x_1,W)/(2k))^{2/3} + \tilde{\varphi}(d(x_3,W)/(2k))^{2/3} \big]^{1/2}\\
&\qquad\times \tilde{\varphi}(|x_2-x_3|/(k+1))^{1/3}\big[\tilde{\varphi}(d(x_2,W)/(2k))^{2/3} + \tilde{\varphi}(d(x_3,W)/(2k))^{2/3} \big]^{1/2} \,d(x_1,x_2, x_3).
\end{align*}
Analogously to \eqref{eqn:proof.quantitative.clt.1}, the above right-hand side is finite and, hence, conditions \eqref{eqn:Integrability1} and \eqref{eqn:Integrability2} are fulfilled for $F$.

By (\ref{eqn:exp.diff.2}), (\ref{eqn:exp.diff.5}), assumption (\ref{eqn:TildeVarphi}) and (\ref{eqn:bound.int.con.3root}) we have
\begin{align*}
&\int \Big[\BE \big(\Delta_{x_1}\widetilde{F}\big)^2 \big(\Delta_{x_2}\widetilde{F}\big)^2\Big]^{1/2}
\Big[\BE \big(\Delta_{x_1,x_3}^2\widetilde{F}\big)^2\big(\Delta_{x_2,x_3}^2\widetilde{F}\big)^2\Big]^{1/2}\,d(x_1,x_2,x_3)\\
&\le \sqrt{C_2C_5} \int \big[ \tilde{\varphi}(d(x_1,W)/k)^{2/3} + \tilde{\varphi}(d(x_2,W)/k)^{2/3} \big]^{1/2}\\
&\quad \times \big[ \tilde{\varphi}(|x_1-x_2|/(2k+1))^{2/3} \tilde{\varphi}(|x_1-x_3|/(2k+1))^{2/3}\\
&\qquad + \tilde{\varphi}(|x_1-x_3|/(2k+1))^{2/3} \tilde{\varphi}(|x_2-x_3|/(2k+1))^{2/3} \big]^{1/2} \, d(x_1,x_2,x_3)\\
&\le c_1 \sqrt{C_2C_5} \int \tilde{\varphi}(d(x,W)/k)^{1/3} \,dx\\
&\le \tilde{c} c_1 \sqrt{C_2C_5} \lambda_d(W)
\end{align*}
with a constant $c_1>0$.
Combining this bound with (\ref{eqn:var.low.bound}), we see that $\gamma_1$ is bounded by some multiple of $\lambda_d(W)^{-1/2}$ for all $W\in\mathcal{K}^d$ with $r(W)\ge\tau$. Analogously, $\gamma_2,\ldots,\gamma_6$ can be treated, where Lemma \ref{lem:4.moment} can be used to bound $\BE F^4$.
This is left to the reader. Finally, the application of Theorem \ref{thm:General.CLT} concludes the proof.
\end{proof}

\section{Total number of components}

As in Section \ref{sasymptoticcov} we let $\eta$ be a stationary Poisson process of intensity $\beta>0$ in $\R^d$ and $\varphi:\R^d\to[0,1]$ be a connection function satisfying \eqref{e5.1}.
So far we have investigated the numbers of components isomorphic to given finite connected graphs. In this section we study the total number of finite components.
For technical reasons we do not count all components with lexicographic minimum in an observation window $W\in\mathcal{K}^d$, but only those whose vertices are all in $W$.
So we define $\bar{\eta}(W)$ as the  number of finite components of $\Gamma(\eta)$ such that all vertices belong to $W$. A strong law of large numbers for $\bar{\eta}(W)$ and related statistics is derived in \cite[Theorem 2]{Penrose91} in case of rectangular observation windows.
The following result is a slightly more general version of Theorem \ref{t1.3}. Recall that $N$ denotes a standard Gaussian random variable.

\begin{theorem}\label{thm:CLTTotal}
Assume that \eqref{e5.1} is satisfied. Then the limit
\begin{equation}\label{eqn:DefinitionLimitVarianceTotalNumber}
\sigma_{\varphi,\varphi} := \lim_{r(W)\to\infty} \frac{\V \bar{\eta}(W)}{\lambda_d(W)}
\end{equation}
exists, is in $(0,\infty)$ and is given by $\sigma_{\varphi,\varphi}=\lim_{m\to\infty} \sum_{i,j=1}^m \sigma_{\varphi,\varphi}^{(i,j)}$.
For $r(W)\to\infty$,
$$
\frac{\bar{\eta}(W)-\BE \bar{\eta}(W)}{\sqrt{\V \bar{\eta}(W)}}\overset{d}{\longrightarrow} N.
$$
\end{theorem}

For the special case of random geometric graphs a similar result as Theorem \ref{thm:CLTTotal} is shown in \cite[Theorem 13.27]{Penrose03}.

For the following lemmas preparing the proof of Theorem \ref{thm:CLTTotal} and the proof itself we can assume without loss of generality that the intensity $\beta$ equals $1$. For $G\in\mathbf{G}$ and $W\in\mathcal{K}^d$ let $\tilde{\eta}_{G}(W)$ be the number of components of $\Gamma(\eta)$ that are isomorphic to $G$ and have only vertices in $W$. Similarly, let $\tilde{\eta}_{k}(W)$, $k\in\N$, be the number of $k$-components of $\Gamma(\eta)$ such that all vertices are in $W$. Note that in \cite{BrugMeester04} $k$-components are counted this way and not as in the previous sections via their lexicographic minima. However, the next lemma and the following corollary show that both ways of counting components that have a given number of vertices or are isomorphic to a given graph are asymptotically equivalent.

\begin{lemma}\label{lem:BoundaryEffects}
Let \eqref{e5.1} be satisfied and let $G\in\mathbf{G}$. Then
$$
\lim_{r(W)\to\infty} \frac{\V (\tilde{\eta}_{G}(W)-\eta_{G}(W))}{\lambda_d(W)}=0.
$$
\end{lemma}

\begin{proof}
Let $G$ have $k\in\N$ vertices and assume that $k\geq 2$ since $\tilde{\eta}_G(W)=\eta_G(W)$ for $k=1$. It follows from the Poincar\'e inequality \eqref{epoincare} and similar arguments as in the proof of Lemma \ref{lem.first.diff} that
\begin{align*}
\V (\tilde{\eta}_{G}(W)-\eta_{G}(W)) &\leq \int \BE (\Delta_x(\tilde{\eta}_{G}(W)-\eta_{G}(W)))^2  \, dx \\
& \leq \int_W \BE \deg(x,\Gamma(\eta_x))^2  \1\big\{x \overset{\leq k}{\llra} W^c \text{ in } \Gamma(\eta_x)\big\} \, dx \\
& \quad +\int_{W^c} \BE \deg(x,\Gamma(\eta_x))^2 \1\big\{x \overset{\leq k}{\llra} W \text{ in } \Gamma(\eta_x)\big\} \, dx.
\end{align*}
Using similar arguments as in the proof of Lemma \ref{lem.diff.bound}, a longer computation yields that the right-hand side can be bounded by
$$
c_{k,\varphi}\int_W \int_{W^c} \varphi(x-y) \, dy \, dx,
$$
where $c_{k,\varphi}>0$ is a constant depending on $k$ and $\varphi$. For any fixed $R>0$ we have that
\begin{align*}
\frac{1}{\lambda_d(W)} \int_W \int_{W^c} \varphi(x-y) \, dy \, dx & \leq \frac{\lambda_d(\{x\in W: d(x, \partial W)\leq R\})}{\lambda_d(W)} m_\varphi\\
& \quad  + \frac{\lambda_d(\{x\in W: d(x, \partial W)\geq R\})}{\lambda_d(W)} \int_{B^d(0,R)^c} \varphi(y) \, dy.
\end{align*}
From \cite[Lemma 3.6 and Lemma 3.7]{HugLastSchulte2016} it follows that the first term on the right-hand side vanishes as $r(W)\to\infty$.
Since the second term tends to zero as $R\to\infty$, we obtain that
$$
\lim_{r(W)\to\infty} \frac{1}{\lambda_d(W)} \int_W \int_{W^c} \varphi(x-y) \, dy \, dx =0,
$$
which completes the proof.
\end{proof}

Combining the $L^2$-convergence from the previous lemma with Theorem \ref{thm:VarianceS} and Theorem \ref{thm:CLTGeneral} leads to the following corollary.

\begin{corollary}\label{cor:BoundaryEffects}
The statements of (a) Theorem \ref{thm:VarianceS} and (b) Theorem \ref{thm:CLTGeneral} are still valid with $\eta_{G_i}(W)$ replaced by $\tilde{\eta}_{G_i}(W)$ for $i\in[m]$.
\end{corollary}

For $W\in\mathcal{K}^d$ and $m\in\N$ we define
$$
\tilde{\eta}_{\leq m}(W):= \sum_{k=1}^m \tilde{\eta}_{k}(W) \quad \text{and} \quad \tilde{\eta}_{>m}(W):= \sum_{k=m+1}^\infty \tilde{\eta}_{k}(W).
$$
Moreover, let
\begin{align*}
q_{\varphi,m} &:= \BP(\text{the sum of the orders of the finite components in } \Gamma(\eta)\\
&\hspace{1.5cm}\text{that are connected with } 0 \text{ in } \Gamma(\eta\cup\{0\}) \text{ is at least } m)
\end{align*}
for $m\in\N$. Note that $q_{\varphi,m}\to0$ as $m\to\infty$.

\begin{lemma}\label{lem:TempBoundsVar}
Assume that \eqref{e5.1} is satisfied and let $\widetilde{C}_\varphi := \BE(\deg(0,\Gamma(\eta\cup\{0\}))+1)^2$ and $C_\varphi:=\BE[\deg(0,\Gamma(\eta\cup\{0\}))^4]^{1/2}$. Then, for all $m,n\in\N$ with $m\leq n$,
$$
\limsup_{r(W)\to\infty} \frac{\V \tilde{\eta}_{\leq m}(W)}{\lambda_d(W)}\leq \widetilde{C}_\varphi, \quad
\limsup_{r(W)\to\infty} \frac{\V (\tilde{\eta}_{\leq m}(W)-\tilde{\eta}_{\leq n}(W))}{\lambda_d(W)}\leq C_\varphi \sqrt{q_{\varphi,m}},
$$
and
$$
\limsup_{r(W)\to\infty} \frac{\V \tilde{\eta}_{>m}(W)}{\lambda_d(W)}\leq C_\varphi \sqrt{q_{\varphi,m}}.
$$
\end{lemma}

\begin{proof}
For $k\in\N$ and $x\in\R^d$ let $B_k(x)$ denote the event that the sum of the orders of the finite components in $\Gamma(\eta)$ that are connected with $x$ in $\Gamma(\eta_x)$ is at least $k$.
The stationarity of $\eta$ implies $\BP(B_k(x))=q_{\varphi,k}$ for $k\in\N$ and $x\in\R^d$.
For $W\in\mathcal{K}^d$ and $x\in W$ we have
\begin{align*}
|\Delta_x \tilde{\eta}_{\leq m}(W)| & \leq \deg(x,\Gamma(\eta_x)) +1, \allowdisplaybreaks \\
|\Delta_x (\tilde{\eta}_{\leq m}(W)-\tilde{\eta}_{\leq n}(W))| & \leq  \deg(x,\Gamma(\eta_x)) \1_{B_m(x)}, \allowdisplaybreaks \\
|\Delta_x\tilde{\eta}_{>m}(W)| & \leq  \deg(x,\Gamma(\eta_x)) \1_{B_m(x)}
\end{align*}
similarly as in Lemma \ref{lem.first.diff}. For $x\in W^c$ all left-hand sides can be bounded by
$$
\sum_{y\in\eta\cap W}\1\{x\lra y \text{ in } \Gamma(\eta_x)\}.
$$
Using the same arguments as in the last step of the proof of Lemma \ref{lem:BoundaryEffects} one can show that
\begin{align*}
\lim_{r(W)\to\infty} &\frac{1}{\lambda_d(W)} \int_{W^c} \BE \bigg(\sum_{y\in\eta\cap W}\1\{x\lra y \text{ in } \Gamma(\eta_x)\}\bigg)^2 \, dx \\
& = \lim_{r(W)\to\infty} \frac{1}{\lambda_d(W)} \int_{W^c} \Bigg( \bigg(\int_W \varphi(y-x) \, dy \bigg)^2 + \int_W \varphi(y-x) \, dy \Bigg)  \, dx =0.
\end{align*}
Now the Poincar\'e inequality \eqref{epoincare} in combination with the Cauchy-Schwarz inequality proves the desired inequalities.
\end{proof}

\begin{lemma}\label{lem:LowerBoundVarianceBall}
If \eqref{e5.1} is satisfied, then
$$
\liminf_{r\to\infty} \frac{\V \bar{\eta}(B^d(0,r))}{r^d}>0.
$$
\end{lemma}

\begin{proof}
For $r>0$ let $g_r$ be a representative of $\bar{\eta}(B^d(0,r))$. In the following we use the same notation and a similar approach as in the proof of Theorem \ref{thm:VarianceS}.

For $(x,t,M)\in\R^d\times[0,1]\times[0,1]^{\mathbb{N}\times\mathbb{N}}$ we denote by $B(x,t,M)$ the event that there are two distinct vertices $y_1,y_2\in B^d(0,r)$ in $\widetilde{\Gamma}(\hat{\eta}_t\cup\{(x,t,M)\})$ which both are only connected to $x$ in $\widetilde{\Gamma}(\hat{\eta}_t\cup\{(x,t,M)\})$ (i.e., each of them has degree one and a single edge to $x$).

It follows from Theorem \ref{thm:VarianceRepresentation} that
\begin{align*}
\V \bar{\eta}(B^d(0,r)) & \geq \int_{B^d(0,r)} \int_0^1 \int \BE\big[ \BE\big[ g_r(T(\hat{\eta}\cup\{(x,t,M)\})) - g_r(T(\hat{\eta}\cup\{(x,t,M)\})\setminus\{x\}) | \hat{\eta}_t \big]^2\\
 & \hspace{3.75cm} \times \mathbf{1}_{B(x,t,M)} \big]   \, \Lambda(dM)\, dt \, dx.
\end{align*}
If a non-isolated vertex is removed, the number of components can not decrease. Hence, we have that
\begin{align*}
& \BE\big[ g_r(T(\hat{\eta}\cup\{(x,t,M)\})) - g_r(T(\hat{\eta}\cup\{(x,t,M)\})\setminus\{x\}) | \hat{\eta}_t \big] \mathbf{1}_{B(x,t,M))} \\
& \leq - \BP\big(\text{$y_1,y_2$ only connected to $x$ in $\widetilde{\Gamma}(\hat{\eta}\cup\{(x,t,M)\})$} | \hat{\eta}_t \big) \mathbf{1}_{B(x,t,M)}.
\end{align*}
Now a short computation proves that
\begin{align*}
& \BP\big(\text{$y_1,y_2$ only connected to $x$ in $\widetilde{\Gamma}(\hat{\eta}\cup\{(x,t,M)\})$}\} | \hat{\eta}_t \big) \mathbf{1}_{B(x,t,M)}\\
& \geq \big(1 - \BP\big(\operatorname{deg}(y_1,\widetilde{\Gamma}(\hat{\eta}\cup\{(x,t,M)\}))\geq 2| \hat{\eta}_t \big) -\BP\big(\operatorname{deg}(y_2,\widetilde{\Gamma}(\hat{\eta}\cup\{(x,t,M)\}))\geq 2 | \hat{\eta}_t \big)\big) \\
& \quad \quad \times \mathbf{1}_{B(x,t,M)}\\
& =\big(1 - (1-\exp(-(1-t)m_\varphi)) - (1-\exp(-(1-t)m_\varphi)) \big) \mathbf{1}_{B(x,t,M)} \\
& =\big(2\exp(-(1-t) m_\varphi) - 1\big) \mathbf{1}_{B(x,t,M)}.
\end{align*}
Choosing $t_0\in [0,1)$ such that $2\exp(-(1-t)m_\varphi) - 1 \geq \frac{1}{2}$ for $t\in[t_0,1]$, we obtain that
$$
\V \bar{\eta}(B^d(0,r)) \geq \frac{1}{4} \int_{B^d(0,r)} \int_{t_0}^1 \int  \mathbb{P}(B(x,t,M)) \, \Lambda(dM)  \, dt \, dx.
$$
For $x\in B^d(0,r)$ and $t\in [t_0,1]$ we have that
\begin{align*}
& \int  \mathbb{P}(B(x,t,M)) \, \Lambda(dM) \\
& = \frac{1}{2}\int \BE \sideset{}{^{\ne}}\sum_{((y_1,t_1,M_1),(y_2,t_2,M_2))\in\hat{\eta}_t^2} \mathbf{1}\{y_1,y_2\in B^d(0,r), \{x,y_1,y_2\} \text{ is a component}\\
& \hskip 6.5cm\text{and $y_1\not\leftrightarrow y_2$}
 \text{ in }\widetilde{\Gamma}(\hat{\eta}_t\cup\{(x,t,M)\})\} \, \Lambda(dM)\\
& = \frac{t^2}{2} \int_{B^d(0,r)^2} \varphi(x-y_1) \varphi(x-y_2) \bar{\varphi}(y_1-y_2) \\
& \hspace{3cm} \times \exp\bigg[t \int \big(\bar{\varphi}(x-y)\bar{\varphi}(y_1-y)\bar{\varphi}(y_2-y) -1\big) \, dy \bigg] \, d(y_1,y_2).
\end{align*}
This implies that
\begin{align*}
\V \bar{\eta}(B^d(0,r)) & \geq \frac{(1-t_0)t_0^2}{8} \int_{B^d(0,r)^3} \varphi(x-y_1) \varphi(x-y_2) \bar{\varphi}(y_1-y_2)\\ & \hspace{3cm} \times \exp\bigg[\int \big(\bar{\varphi}(x-y)\bar{\varphi}(y_1-y)\bar{\varphi}(y_2-y) -1\big) \, dy \bigg] \, d(x,y_1,y_2).
\end{align*}
Consequently we have that
\begin{align*}
& \liminf_{r\to\infty} \frac{\V \bar{\eta}(B^d(0,r))}{\lambda_d(B^d(0,r))}\\
 & \geq \frac{(1-t_0)t_0^2}{8} \int \varphi(y_1) \varphi(y_2) \bar{\varphi}(y_1-y_2)  \exp\bigg[\int \big(\bar{\varphi}(y)\bar{\varphi}(y_1-y)\bar{\varphi}(y_2-y) -1\big) \, dy \bigg] \, d(y_1,y_2).
\end{align*}
Next we show
\begin{equation} \label{eqn:IntegralVarphis}
\int \varphi(y_1) \varphi(y_2) \bar{\varphi}(y_1-y_2) \, d(y_1,y_2) >0,
\end{equation}
which completes the proof. If $\lambda_d(\{z\in \R^d: \varphi(z)=1\})=0$, this is obviously true. Otherwise, one can choose a $r_0\in (0,\infty)$ such that
$$
0<\lambda_d(\{z\in B^d(0,r_0)^c: \varphi(z)=1\}) < \frac{1}{2} \lambda_d(\{z\in \R^d: \varphi(z)=1\}).
$$
Together with
\begin{align*}
\lambda_d(\{y\in \R^d: \varphi(y)=1\}\setminus B^d(x,r_0)) & \geq \lambda_d(\{ y\in \R^d: \varphi(y)=1, \langle x,y \rangle\leq 0 \}) \\
& = \frac{1}{2} \lambda_d(\{ y\in \R^d: \varphi(y)=1 \})
\end{align*}
for $x\in B^d(0,r_0)^c$, this yields that
$$
\lambda_d^2(\{(y_1,y_2)\in B^d(0,r_0)^c\times \R^d : \varphi(y_1)=1, \varphi(y_2)=1, \varphi(y_1-y_2)\neq 1\})>0
$$
and proves \eqref{eqn:IntegralVarphis}.
\end{proof}

\begin{proof}[Proof of Theorem \ref{thm:CLTTotal}]
For $m\in\N$ Corollary \ref{cor:BoundaryEffects} (a) yields
$$
\sigma_{\varphi,\leq m}:=\lim_{r(W)\to\infty} \frac{\V \tilde{\eta}_{\leq m}(W)}{\lambda_d(W)} = \sum_{i,j=1}^m \sigma_{\varphi,\varphi}^{(i,j)}.
$$
We have that, for $m,n\in\N$ with $m\leq n$,
\begin{align*}
& \big| \sigma_{\varphi,\leq m} - \sigma_{\varphi,\leq n}\big| \\
& = \lim_{r(W)\to\infty} \frac{1}{\lambda_d(W)}| \V \tilde{\eta}_{\leq m}(W) - \V \tilde{\eta}_{\leq n}(W)| \allowdisplaybreaks \\
& = \lim_{r(W)\to\infty} \frac{1}{\lambda_d(W)} \Big|\sqrt{\V \tilde{\eta}_{\leq m}(W)} + \sqrt{\V \tilde{\eta}_{\leq n}(W)}\Big| \, \Big|\sqrt{\V \tilde{\eta}_{\leq m}(W)} - \sqrt{\V \tilde{\eta}_{\leq n}(W)}\Big| \allowdisplaybreaks \\
& \leq \lim_{r(W)\to\infty} \frac{1}{\lambda_d(W)} \Big|\sqrt{\V \tilde{\eta}_{\leq m}(W)} + \sqrt{\V \tilde{\eta}_{\leq n}(W)}\Big| \, \sqrt{\V (\tilde{\eta}_{\leq m}(W) - \tilde{\eta}_{\leq n}(W))} \\
& \leq 2 \widetilde{C}_\varphi^{1/2} C_\varphi^{1/2} q_{\varphi,m}^{1/4},
\end{align*}
where we used the triangle inequality in $L^2(\BP)$ and Lemma \ref{lem:TempBoundsVar}. Since $q_{\varphi,m}\to0$ as $m\to\infty$, $(\sigma_{\varphi,\leq m})_{m\in\N}$ is a Cauchy sequence. Thus the limit
$$
\tilde{\sigma}_{\varphi}:= \lim_{m\to\infty} \sigma_{\varphi,\leq m} = \lim_{m\to\infty} \sum_{i,j=1}^m \sigma_{\varphi,\varphi}^{(i,j)}
$$
exists and is finite.

It follows from the triangle inequality in $L^2(\BP)$ and Lemma \ref{lem:TempBoundsVar} that, for $m\in\N$,
\begin{align*}
& \limsup_{r(W)\to\infty} \frac{\big| \sqrt{\V \bar{\eta}(W)} -  \sqrt{\V \tilde{\eta}_{\leq m}(W)} \big|}{\sqrt{\lambda_d(W)}}\\
& \leq \limsup_{r(W)\to\infty} \frac{\sqrt{\V (\bar{\eta}(W) - \tilde{\eta}_{\leq m}(W))} }{\sqrt{\lambda_d(W)}}  = \limsup_{r(W)\to\infty} \frac{\sqrt{\V \tilde{\eta}_{> m}(W)}}{\sqrt{\lambda_d(W)}}  \leq C_\varphi^{1/2} q_{\varphi,m}^{1/4}.
\end{align*}
Since $\lim_{r(W)\to\infty} \lambda_d(W)^{-1} \V \tilde{\eta}_{\leq m}(W)=\sigma_{\varphi,\leq m}$ for any $m\in\N$ and $q_{\varphi,m} \to 0$ as $m\to\infty$, we obtain that the limit $\sigma_{\varphi,\varphi}$ in \eqref{eqn:DefinitionLimitVarianceTotalNumber} exists and equals $\tilde{\sigma}_\varphi$. Moreover, Lemma \ref{lem:LowerBoundVarianceBall} yields $\sigma_{\varphi,\varphi}>0$.

Let $h:\R\to\R$ be a function with Lipschitz constant at most one. For any $m\in\N$ and $W\in\mathcal{K}^d$, the triangle inequality implies
$$
\bigg| \BE h\bigg( \frac{\bar{\eta}(W)- \BE \bar{\eta}(W)}{\sqrt{\V \bar{\eta}(W)}} \bigg) - \BE h(N) \bigg| \leq U_1+U_2+U_3
$$
with
\begin{align*}
U_1 & := \bigg| \BE h\bigg( \frac{\bar{\eta}(W)- \BE \bar{\eta}(W)}{\sqrt{\V \bar{\eta}(W)}} \bigg) - \BE h\bigg( \frac{\tilde{\eta}_{\leq m}(W)- \BE \tilde{\eta}_{\leq m}(W)}{\sqrt{\V \bar{\eta}(W)}} \bigg)  \bigg|,  \allowdisplaybreaks \\
U_2 & :=  \bigg| \BE h\bigg( \frac{\tilde{\eta}_{\leq m}(W)- \BE \tilde{\eta}_{\leq m}(W)}{\sqrt{\V \bar{\eta}(W)}} \bigg) - \BE h\bigg( \frac{\tilde{\eta}_{\leq m}(W)- \BE \tilde{\eta}_{\leq m}(W)}{\sqrt{\V \tilde{\eta}_{\leq m}(W)}} \bigg)  \bigg|, \allowdisplaybreaks \\
U_3 & :=  \bigg| \BE h\bigg( \frac{\tilde{\eta}_{\leq m}(W)- \BE \tilde{\eta}_{\leq m}(W)}{\sqrt{\V \tilde{\eta}_{\leq m}(W)}} \bigg) - \BE h(N) \bigg|.
\end{align*}
From Corollary \ref{cor:BoundaryEffects} (b) we know that, for $r(W)\to\infty$, $\tilde{\eta}_{\leq m}(W)$ satisfies a central limit theorem and, thus, $U_3\to 0$. Using the Lipschitz property of $h$, the Jensen inequality and Lemma \ref{lem:TempBoundsVar}, we see that
\begin{align*}
\limsup_{r(W)\to\infty} U_1 & \leq \limsup_{r(W)\to\infty} \frac{\BE |\bar{\eta}(W) - \BE \bar{\eta}(W) -(\tilde\eta_{\leq m}(W)-\BE \tilde\eta_{\leq m}(W))|}{\sqrt{\V \bar{\eta}(W)}}\\
& \leq \limsup_{r(W)\to\infty} \frac{\sqrt{\V (\bar{\eta}(W)-\tilde\eta_{\leq m}(W))}}{\sqrt{\V \bar{\eta}(W)}} = \limsup_{r(W)\to\infty} \frac{\sqrt{\V \tilde\eta_{>m}(W)}}{\sqrt{\V \bar{\eta}(W)}} \leq \frac{C_\varphi^{1/2}}{\sqrt{\sigma_{\varphi,\varphi}}} q_{\varphi,m}^{1/4}.
\end{align*}
Again, by the Lipschitz property of $h$ and the Jensen inequality, we have
$$
\limsup_{r(W)\to\infty} U_2 \leq \limsup_{r(W)\to\infty} \bigg|1- \frac{\sqrt{\V \tilde{\eta}_{\leq m}(W)}} {\sqrt{\V \bar{\eta}(W)}} \bigg| \, \BE\bigg|  \frac{\tilde{\eta}_{\leq m}(W)- \BE \tilde{\eta}_{\leq m}(W)}{\sqrt{\V \tilde{\eta}_{\leq m}(W)}}\bigg| \leq  \bigg|1-\frac{\sqrt{\sigma_{\varphi,\leq m}}}{\sqrt{\sigma_{\varphi,\varphi}}} \bigg|.
$$
Since $q_{\varphi,m}\to0$ and $\sigma_{\varphi,\leq m}/\sigma_{\varphi,\varphi}\to1$ as $m\to\infty$, letting first $r(W)\to\infty$ and then $m\to\infty$ yields
$$
\lim_{r(W)\to\infty} \BE h\bigg( \frac{\bar{\eta}(W)- \BE \bar{\eta}(W)}{\sqrt{\V \bar{\eta}(W)}} \bigg) = \BE h(N),
$$
which completes the proof.
\end{proof}

\appendix

\section{A variance representation for Poisson functionals}

In this appendix we derive a variance representation for functionals of Poisson processes with birth times in terms of difference operators and conditional expectations. The proof heavily relies on some results from \cite{LastPenrose2011}.

Let $(\mathbf{Y},\mathcal{Y})$ be a measurable space with a $\sigma$-finite measure $\mu$ and let $\eta'$ be a Poisson process on $\mathbf{Y}\times [0,1]$ with intensity measure $\mu\otimes \lambda_1|_{[0,1]}$, where $\lambda_1|_{[0,1]}$ denotes the restriction of the Lebesgue measure to the unit interval. For $t\in[0,1]$ let $\eta_t'$ be the restriction of $\eta'$ to $\mathbf{Y}\times [0,t)$. Recall that $\mathbf{N}(\mathbf{Y}\times[0,1])$ is the set of $\sigma$-finite counting measures on $\mathbf{Y}\times [0,1]$, which is equipped with the smallest $\sigma$-field such that the maps $\mathbf{N}(\mathbf{Y}\times [0,1])\ni\nu\mapsto \nu(A)$ are measurable for all measurable $A\subset \mathbf{Y}\times [0,1]$.

\begin{theorem}\label{thm:appendix.variance.repre}
Let $F=f(\eta')$ with a measurable $f: \mathbf{N}(\mathbf{Y}\times [0,1]) \to \R$ be such that $\BE F^2<\infty$. Then
$$
\V F = \int_0^1 \int \BE[ \BE[ D_{(x,t)}F | \eta_t']^2]  \, \mu(dx) \, dt.
$$
\end{theorem}

\begin{proof}
Let us use the short-hand notation $h(x,t):= \BE[ D_{(x,t)}F | \eta_t']$ for $(x,t)\in \mathbf{Y}\times[0,1]$.
We equip $\mathbf{Y}\times [0,1]$ with the order $(x_1,t_1)\prec (x_2,t_2)$ if and only if $t_1<t_2$ so that we are in the framework of Section 2 in \cite{LastPenrose2011}. Now  \cite[Theorem 2.1]{LastPenrose2011} implies that
$$
\int_0^1 \int \BE h(x,t)^2 \, \mu(dx) \, dt = \int_0^1 \int \BE[ \BE[ D_{(x,t)}F | \eta_t']^2]  \, \mu(dx) \, dt<\infty \quad \text{and} \quad F- \BE F = \delta(h),
$$
where $\delta$ is the so-called Kabanov-Skorohod integral (see the first display on p.\ 1591 in \cite{LastPenrose2011}). Then it follows from \cite[Corollary 2.7]{LastPenrose2011} that
$$
\V F = \BE \delta(h)^2 = \int_0^1 \int \BE h(x,t)^2 \, \mu(dx) \, dt,
$$
which completes the proof.
\end{proof}

\bigskip
\noindent
{\bf Acknowledgement:} We acknowledge the support of the German Science Foundation (DFG) through the research group ``Geometry and Physics of Spatial Random Systems'' (GPSRS, FOR 1548).

\end{document}